%% file: Oriented_Location-Domination.tex
\documentclass{article}
\usepackage{amssymb,amsthm,amsmath}
\usepackage{a4wide}
\usepackage[T1]{fontenc}
\usepackage[utf8]{inputenc}         
\usepackage{authblk}
\usepackage{float}

\usepackage{graphicx} 
\usepackage{enumerate}
\usepackage{enumitem, hyperref}
\usepackage{color}
\usepackage{tikz}
\usetikzlibrary{shapes}
\theoremstyle{plain}
\newtheorem{theorem}{Theorem}

\newtheorem{lemma}[theorem]{Lemma}
\newtheorem{corollary}[theorem]{Corollary}
\theoremstyle{definition}

\newtheorem{remark}[theorem]{Remark}

\newtheorem{conjecture}[theorem]{Conjecture}
\newtheorem{problem}[theorem]{Open problem}
\newtheorem{claim}[theorem]{Claim}

\usepackage{thmtools}
\usepackage{thm-restate}

\newcommand{\decisionpb}[4]{
        \begin{minipage}{#4\textwidth}
                #1\\
                \emph{Instance:} #2\\ 
                \emph{Question:} #3
        \end{minipage}
}
\newcommand{\PBLD}{\textsc{Locating-Dominating-Set}}
\newcommand{\PBOLD}{\textsc{Lower-Directed-LD-Number}}
\newcommand{\PBD}{\textsc{Dominat\-ing-Set}}
\newcommand{\old}{\overset{\rightarrow}{\gamma}_{LD}}
\newcommand{\mold}{\overset{\rightarrow}{\Gamma}_{LD}}
\newcommand{\ld}{{\gamma}_{LD}}
\newcommand{\LD}{{\Gamma}_{d}}
\tikzstyle{vertexstyle}=[circle, fill, minimum size=5, inner sep =0]
\tikzstyle{edgestyle}=[draw, line width=1.2pt]

\newcommand{\SL}{\mathcal{SL}}

\title{Locating-dominating sets: from graphs to oriented graphs\thanks{This work was supported by ANR project GrR (ANR-18-CE40-0032)}}

\author[1]{Nicolas Bousquet}
\author[1]{Quentin Deschamps}
\author[1,2]{Tuomo Lehtil\"a\footnote{Research supported by the Finnish cultural foundation and by the Academy of Finland grant 338797}}
\author[1]{Aline Parreau}

\affil[1]{Universit\'e de Lyon, Universit\'e Lyon 1, LIRIS UMR CNRS 5205, F-69621, Lyon, France}
\affil[2]{Department of mathematics and statistics, University of Turku, Finland}

\date{\today}

\begin{document}

\maketitle

\begin{abstract}
A locating-dominating set of an undirected graph is a subset of vertices $S$ such that $S$ is dominating and for every $u,v \notin S$, the neighbourhood of $u$ and $v$ on $S$ are distinct (i.e. $N(u) \cap S \ne N(v) \cap S$). 
Locating-dominating sets have received a considerable attention in the last decades.
In this paper, we consider the oriented version of the problem. A locating-dominating set in an oriented graph is a set $S$ such that for each $w\in V\setminus S$, $N^-(w)\cap S\neq\emptyset$ and for each pair of distinct vertices $u,v \in V \setminus S$, $N^-(u) \cap S \ne N^-(v) \cap S$. We consider the following two parameters. Given an undirected graph $G$, we look for $\old(G)$ ($\mold(G))$ which is the size of the smallest (largest) optimal locating-dominating set over all orientations of $G$.
In particular, if $D$ is an orientation of $G$, then $\old(G)\leq\ld(D)\leq\mold(G)$ where $\ld(D)$ is the minimum size of a locating-dominating set of $D$.

For the best orientation, we prove that, for every twin-free graph $G$ on $n$ vertices, $\old(G) \le n/2$ which proves a ``directed version'' of a widely studied conjecture on the location-domination number.  As a side result we obtain a new improved upper bound for the location-domination number in undirected trees.
Moreover, we give some bounds for $\old(G)$ on many graph classes and drastically improve the value $n/2$ for (almost) $d$-regular graphs by showing that $\old(G) \in O(\log d / d \cdot n)$ using a probabilistic argument.

While $\old(G)\leq\ld(G)$ holds for every graph $G$, we give some graph classes such as outerplanar graphs for which  $\mold(G)\geq\ld(G)$  and some for which $\mold(G)\leq \ld(G)$ such as complete graphs. We also give general bounds for  $\mold(G)$ such as $\mold(G)\geq\alpha(G)$. Finally, we show that for many graph classes $\mold(G)$ is polynomial on $n$ but we leave open the question whether there exist graphs with $\mold(G)\in O(\log n)$.
\end{abstract}

\section{Introduction}

A {\em dominating set} of an undirected graph $G$ is a subset $S$ of its vertices such that each vertex of $G$ not in $S$ has a neighbour in $S$. The {\em domination number} of $G$, denoted by $\gamma(G)$ is the size of a smallest dominating set of $G$. Domination theory is one of the main topics of graph theory, see for example the two reference books \cite{haynes1998domination, haynes1998fundamentals}.
Among the variations of domination, the {\em location-domination}, introduced by Slater \cite{slater1987domination}, has been extensively studied. A {\em locating-dominating set} of an undirected graph $G$ is a dominating set $S$ such that all vertices not in $S$ have pairwise distinct neighbourhoods in $S$. The {\em location-domination number} of $G$, denoted by $\ld(G)$, is the size of a smallest locating-dominating set of $G$. Since $V(G)$ is always a locating-dominating set, $\ld(G)$ is well-defined. Structural and algorithmic properties of locating-dominating sets have been widely studied  (see e.g.~\cite{lobstein2012watching} for an online bibliography).
Location-domination in directed graphs was briefly mentioned in several articles (see e.g.~\cite{charondirected, skaggs2007identifying}) and further studied in \cite{foucaud2020domination}. A \emph{locating-dominating set} of a directed graph $D$ is a subset $S$ of its vertices such that two vertices not in $S$ have distinct and non-empty {\em in-neighbourhoods} in $S$. The {\em directed location-domination number} of $D$, denoted by $\ld(D)$, is the size of a smallest locating-dominating set of $D$.

Two oriented graphs with the same underlying graph can have a very different behaviour towards locating-dominating sets. Let us illustrate it on \emph{tournaments} that are oriented complete graphs.
Transitive tournaments (i.e. acyclic tournaments) have directed location-domination number $\lceil n /2 \rceil$ whereas one can construct locating-dominating sets of size $\lceil \log n \rceil$ for a well-chosen orientation of $K_n$ \cite{skaggs2007identifying}.
Following the idea of Caro and Henning for domination \cite{caro2012directed} and the work started by Skaggs \cite{skaggs2007identifying}, we study in this paper the best and worst orientations of a graph for locating-dominating sets. Orientation of graph $G$ is considered to be \textit{best} (resp. \textit{worst}) if it minimizes (resp. maximizes) the location-domination number over all the orientations of $G$.
A similar line of work has been recently initiated for the related concepts of identifying codes \cite{cohen2018minimum} and metric dimension \cite{OrientedMD}. 

The two parameters that are be considered in this paper are the following. The {\em lower directed location-domination number} of an undirected graph $G$, denoted by $\old(G)$, is the minimum directed location-domination number over all the orientations of $G$. 
The {\em upper directed location-domination number} of an undirected graph $G$, denoted by $\mold(G)$, is the maximum directed location-domination number over all the orientations of $G$.

\subsection*{Outline of the paper}

Basic definitions, some background and first results are given in Section \ref{sec:preliminaries}. Section~\ref{sec:best} is dedicated to the study of the best orientations whereas Section \ref{sec:worst} focuses on the worst orientations.

\paragraph{Main results on best orientations.}
We first give basic results on $\old(G)$ and relations with classical parameters of graphs. Skaggs~\cite{skaggs2007identifying} proved in 2007 that for any graph $G$, $\old(G)\leq \ld(G)$. We refine this inequality by proving that, in graphs without cycles of size 4 (as a subgraph), $\old(G)$ and $\ld(G)$ coincide. As a consequence, computing $\old(G)$ is NP-complete.

Two vertices are \emph{twins} if they have the same open or closed neighbourhood. Twins play an important role in locating-dominating sets since any locating-dominating set must contain at least one vertex of each pair of twins. As a consequence, if $G$ is a star on $n$ vertices, then $\old(G)=n-1$.
In Section \ref{sec:n/2best}, we prove that this function can be drastically improved when the graph $G$ is twin-free, which is one of the main contributions of our paper.

\begin{restatable}{theorem}{thmtwinfreeRS}
\label{thm:n/2}
Let $G$ be a twin-free graph of order $n$ with no isolated vertices. Then, $\old(G)\leq n/2$.
\end{restatable}

The fact that any twin-free graph of order $n$ satisfies $\ld(G) \le n/2$ is a notorious conjecture, left open in~\cite{foucaudtwinfree,garijoconjecture} for instance.

\begin{conjecture}[\cite{garijoconjecture}]\label{conj:twinfree}
If $G$ is a twin-free graph of order $n$, then $\ld(G)\le n/2$.
\end{conjecture}

The proof of Theorem \ref{thm:n/2} holds in two steps. First, we show in Section \ref{sec:subgraph} that $\old(G)$ is the smallest undirected location-domination number among all the (connected) spanning subgraphs of $G$. Then, we prove in Section \ref{sec:n/2best} that there exists a spanning subgraph for which the condition is satisfied. In particular, our result implies a weakening of Conjecture \ref{conj:twinfree} since we prove that any twin-free connected graph $G$ on $n$ vertices admits a spanning subgraph $H$ with $\ld(H)\leq n/2$. As a side result we obtain a new improved upper bound for the location-domination number in trees. We also give a characterization for trees attaining this new upper bound.

We then focus on (almost) regular graphs in Section~\ref{sec:regular} and prove, using a probabilistic argument, that there exists a constant $c_d$ such that, if $G$ is $d$-regular, 
$$\old(G)\leq c_d \cdot \frac{\log d}{d} \cdot |V(G)|.$$
We continue this subsection by giving some bounds using independence and matching numbers.

\paragraph{Main results on worst orientations.}

In Section \ref{sec:worstbasic}, we give some examples and relate $\mold(G)$ with some classical graph parameters. In particular, we prove that $\mold(G)\geq \ld(G)$ if $G$ does not have any cycle of length $4$ as a (not necessarily induced) subgraph. 
Moreover we prove that if $G$ is a $C_4$-free bipartite graph (which in particular, contains the class of trees), then $\mold(G)=\alpha(G)$ where $\alpha(G)$ is the independence number of $G$.

In Section \ref{SecLowerboundLD}, we prove that $\mold(G)\geq \ld(G)$ is satisfied for other graph classes such as bipartite graphs, cubic graphs, and outerplanar graphs. Somehow surprisingly at first glance, $\mold(G)\geq \ld(G)$ is not always true. In \cite{foucaud2020domination}, Foucaud et al. have shown that for a complete graph $K_n$ on $n$ vertices we have $\mold(K_n)=\lceil n/2\rceil$ but $\ld(K_n)=n-1$. We prove that the existence of twins is not the reason why this inequality fails since we exhibit a family of twin-free graphs for which the ratio $\mold(G)/\ld(G)$ tends to $1/2$. We did not succeed to bound this ratio by a constant. However, we prove that $\mold(G)\geq\ld(G)/\lceil\log_2(\Delta(G))+1\rceil$. We leave  the existence of a constant bounding $\mold(G)/\ld(G)$ as an open problem. 

Finally, in Section \ref{sec:forbidden}, we provide some lower bounds on $\mold(G)$ using the number of vertices. For numerous classes of graphs, we actually have $\mold(G)\geq c_1\cdot n^{c_2}$ where $c_1$ and $c_2$ are constant. This is true for perfect graphs (with $c_2=1/2$), $C_3$-free graphs, claw-free graphs and actually for any $\chi$-bounded class of graphs with a polynomial $\chi$-bounding function. However, we leave  as an open problem the existence of a graph $G$ on $n$ vertices such that $\mold(G)$ is logarithmic on $n$.

Note that we did not find the complexity of computing $\mold(G)$. In particular, it is not clear that this problem belongs to NP.

\section{Preliminaries}\label{sec:preliminaries}
\subsection{Notations}

We give in this subsection the main definitions and notations we are using along the paper. The reader may refer to some classical graph theory books like \cite{Bondy1976} for missing definitions.

Let $G=(V,E)$ be an undirected and simple graph. We usually denote by $n$ the number of vertices of $G$.
We denote by $N_G(u)$ (or $N(u)$ when $G$ is clear from context) the \emph{open neighbourhood} of a vertex $u$, that is the set of neighbours of $u$. And we denote by $N_G[u]$ (abbreviated into $N[u]$) the \textit{closed neighbourhood} of $u$ that is $N(u) \cup \{ u \}$.  Two vertices $u$ and $v$ are {\em twins} if $N(u)=N(v)$ or $N[u]=N[v]$. The {\em degree} of a vertex $u$, denoted by $d(u)$, is the size of $N(u)$. The minimum and maximum degree of $G$ are respectively denoted by $\delta(G)$ and $\Delta(G)$. A {\em leaf} is a vertex of degree $1$.

A graph $H$ is a {\em subgraph} of $G$ if $V(H)\subseteq V(G)$ and $E(H)\subseteq E(G)$. A subgraph $H$ is {\em induced} if for any pair of vertices of $H$, $(x,y)$ is an edge of $H$ if and only if it is an edge of $G$. A graph $G$ is \emph{$H$-free} if it does not contain $H$ has an induced subgraph. We say that a graph $G$ is \emph{without $H$} if $G$ does not contain $H$ as a subgraph (not necessarily induced). 
A subgraph $H$ is a {\em spanning} subgraph if $V(H)=V(G)$.

The complete graph on $n$ vertices is denoted by $K_n$. The complete bipartite graph with size $n$ and $m$ is denoted by $K_{n,m}$. A \textit{star} is a graph isomorphic to $K_{1,m}$. The star with three leaves, $K_{1,3}$, is also called a {\em claw}. The cycle on $n$ vertices is denoted by $C_n$ whereas the path on $n$ vertices is denoted by $P_n$.
The \textit{girth} of a graph $G$ is the length of a shortest cycle in $G$. If $G$ does not contain any cycle we say that $G$ has infinite girth.
A set $S$ of vertices is \textit{independent} if they are pairwise non-adjacent. A set $S$ is an \textit{edge cover} if every edge has at least one endpoint in $S$. A set of edges $M$ is a \textit{matching} if no two edges in $M$ share an endpoint. In a graph $G$, we denote the cardinalities of maximum independent sets and matchings by $\alpha(G)$ and $\alpha'(G)$, respectively. Moreover, the cardinality of a minimum edge cover is denoted by $\beta(G)$. The clique number of a graph $G$, denoted by $\omega(G)$, is the maximal order of a complete subgraph of $G$.

Let $S$ be a subset of $V$. Set $S$ is a {\em dominating set} of $G$ if any vertex of $G$ is either in $S$ or adjacent to a vertex of $S$. The minimum size of a dominating set is denoted by $\gamma(G)$.
 We denote by $I_G(S;u)$ 
($I(u)$ for short) the set $N_G(u)\cap S$ that is the neighbours in $S$ of a vertex $u$.  Note that $S$ is a \textit{locating-dominating set} if for each vertex $u\in V(G)\setminus S$, $I(u)$ is non-empty (since $S$ is a dominating set) and for each pair of distinct vertices $u,v\in V\setminus S$, we have $I(u)\neq I(v)$. We say that a vertex $s\in S$ \emph{separates} $u$ and $v$ if $s$ is in exactly one of sets $I(u)$ and $I(v)$. Note that any locating-dominating set must intersect any pair of twins. The minimum size of a locating-dominating set of $G$ is denoted by $\ld(G)$.

\

These notions are similarly defined for directed graphs.
In this paper, we mostly consider directed graphs derived from orienting an undirected graph. A \textit{directed graph} (also called \textit{digraph}), is a pair $D = (V,E)$, where $V$ is a set whose
elements are called vertices, and $A$ is a set of ordered pairs of vertices, called \textit{arcs}. Let $G=(V,E)$ be a simple undirected graph. An \textit{orientation} of $G$ is a directed graph (\textit{oriented graph}) $D$ on $V$ where every edge $uv$ of $G$ is either oriented from $u$ to $v$ (resulting to the arc $(u,v)$ in $D$) or from $v$ to $u$ (resulting to the arc $(v,u)$).  In particular, all the directed graphs considered are oriented and simple: if $(u,v)$ is an arc then $(v,u)$ is not. The undirected graph $G$ is called the {\em underlying} graph of $D$. Unless otherwise stated, ``graph'' means ``undirected graph''. A {\em tournament} is an orientation of a complete graph.
The open out-neighbourhood and in-neighbourhood of a vertex $u$ of $D$ are denoted by $N_D^+(u)$ and $N_D^-(u)$ whereas the closed out- and in-neighbourhood are denoted by $N_D^+[u]$ and $N_D^-[u]$. The maximum out- and in-degree are denoted by $\Delta^+(G)$ and $\Delta^-(G)$.
A {\em source} is a vertex with no in-neighbours.
Locating dominating sets are defined similarly as in the undirected case by considering the in-neighbourhoods. We denote by $I_D(S;u)$ (or $I(u)$ for short) the set $N_D^-(u)\cap S$, that is,  the in-neighbours of $u$ that are in a set $S$ of vertices. The set $S$ is a locating-dominating set of $D$ if all the sets $I_D(S;u)$ are non-empty and distinct for $u\notin S$. The minimum size of a locating-dominating set of $D$, called the minimum directed location-domination number, is denoted by $\ld(D)$.

We finally recall the two main parameters that we are considering along this paper. The {\em lower directed location-domination number} of an undirected graph $G$, denoted by $\old(G)$, is the minimum directed location-domination number over all the orientations of $G$. Formally, we have
$$\old(G)=\min\{\ld(D)\mid D \text{ is an orientation of } G\}.$$
The {\em upper directed location-domination number} of an undirected graph $G$, denoted by $\mold(G)$, is the maximum directed location-domination number over all the orientations of $G$. Formally, we have
$$\mold(G)=\max\{\ld(D)\mid D \text{ is an orientation of } G\}.$$


\subsection{Preliminary results and examples}

Let $D$ be a digraph and $u$ be a non-source vertex of $D$. Then, $V(D)\setminus\{u\}$ is a locating-dominating set of $D$. In particular, for any directed graph containing at least one edge, $\LD(D)\leq n-1$. 
In \cite{foucaud2020domination}, the authors characterized those digraphs reaching this extremal value. This characterization is useful for studying the extremal values of $\old(G)$ and $\mold(G)$. A {\em directed star} is a (non-necessarily simple) directed graph such that the underlying graph is a star. A {\em bi-directed clique} is a directed graph that contains all the possible arcs between two vertices.

\begin{theorem}[\cite{foucaud2020domination}, Theorem $6$]\label{Digraphn-1}
Let $D$ be a connected (non necessarily simple) digraph of order $n \ge 2$. Then, $\ld(D) = n - 1$ if and only if at least one of the following conditions holds:
\begin{enumerate}
    \item $n=3$;
    \item $D$ is a directed star;
    \item $V(D)$ can be partitioned into three (possibly empty) sets $S_1$, $C$ and $S_2$, where $S_1$ and $S_2$ are independent sets, $C$ is a bi-directed clique, and the remaining arcs in $D$ are all the possible arcs from $S_1$ to $C \cup S_2$ and those from $C$ to $S_2$.
\end{enumerate}
\end{theorem}

In particular, any orientation of a star has location-domination number $n-1$.

\begin{corollary}\label{cor:star}
Let $G$ be a star on $n$ vertices. Then, $\old(G)=\mold(G)=n-1$.
\end{corollary}

In \cite{foucaud2020domination}, the authors also proved a tight upper bound for tournament:

\begin{theorem}[\cite{foucaud2020domination}]\label{thm:tournament}
Let $D$ be a tournament on $n$ vertices. Then, $\ld(D)\leq \lceil n/2 \rceil$. Moreover, $\ld(D)=\lceil n/2\rceil$ if $D$ is transitive.
\end{theorem}

As a consequence, the upper directed location-domination number of complete graphs is known:

\begin{corollary}\label{cor:complete}
Let $n\geq 2$ be an integer. Then, $\mold(K_n)=\lceil n/2 \rceil$.
\end{corollary}

Concerning the best orientation of a complete graph, Skaggs  proved in his thesis \cite{skaggs2007identifying} that one can obtain the best possible number for $\old(G)$. For the sake of completeness, we add a short proof of this result.

\begin{theorem}[\cite{skaggs2007identifying}, Proposition 5.4]\label{thm:clique}
Let $n\geq 2$ be an integer. Let $k$ be the smallest integer such that $n\leq k+2^k-1$. Then, $\old(K_n)=k$.
\end{theorem}

\begin{proof}
Let $S$ be a set of $k$ vertices of $K_n$.
Then, consider an injective map $f$ from the other vertices of $K_n$ (there are at most $2^k-1$ of them) to the non-empty subsets of $S$. 
Let $u\notin S$ and $v\in S$.  Orient edge $uv$ from $v$ to $u$ if $v\in f(u)$ and from $u$ to $v$ otherwise. Orient all the other edges in any direction.
Then, $S$ is a locating-dominating set for this orientation of $K_n$.
\end{proof}

\section{Best orientation}\label{sec:best}
In this section we focus on the best orientation. We first give basic results and links with classical parameters. Then, we give another definition of $\old(G)$ using spanning subgraphs and use this definition to show that $\old(G)\leq n/2$ if $G$ is twin-free. We finally improve this last result in the case of almost regular graphs.

\subsection{Basics}

\begin{theorem}\label{BestDirLDUppBound}
Let $G$ be a graph of order $n$. Then, 
\begin{enumerate}
    \item  \cite[Proposition 5.3]{skaggs2007identifying} $\old(G)\leq \ld(G)$.
    \item $\old(G)\leq n-\alpha'(G)$.
\end{enumerate}
\begin{proof}
Claim $(1)$ is proved in \cite{skaggs2007identifying}, for completeness, we include a short proof here. Consider a graph $G$ and a locating dominating set $S$ of size $\ld(G)$ of $G$. Then, orient all the edges $uv$ between $S$ and $V\setminus S$ from $S$ to $V\setminus S$ and all the other edges in any way. Then, $S$ is a locating-dominating set for this orientation.

Let us next prove $(2)$. Let $G$ be a graph on $n$ vertices and let $M$ be a maximum matching of $G$. Let $V_M$ be a subset of vertices containing exactly one vertex from each edge of $M$ and $C_M$ be the set of vertices which are not endpoints of edges in $M$. Let $C=V_M\cup C_M$. Note that $|C|=n-\alpha'(G)$. Choose any orientation $D'$ of $G$ where the edges in $M$ have their tails in $C$ and all the other edges between $V\setminus C$ and $C$ are oriented from $V\setminus C$ to $C$.
Now, $C$ is a locating-dominating set in $D'$ since all the vertices of $V\setminus C$ have exactly one in-neighbour in $V_M$ and all of them are pairwise distinct.
\end{proof}
\end{theorem}

We show that these bounds are tight in Corollary \ref{extremalOLD} and Theorem \ref{c4freeBestLD}.
Using Theorem~\ref{Digraphn-1}, we next provide a characterization of graphs reaching the extremal value $\old(G)=n-1$.

\begin{corollary}\label{extremalOLD}
For any connected graph $G$ of order $n \ge 2$, $\old(G)=n-1$ if and only if either $n=3$ or $G$ is a star.
\end{corollary}
\begin{proof}
Let $G$ be a graph of order $n\geq 2$ with $\old(G)=n-1$. If either $n=3$ or $G$ is a star, then, $\old(G)=n-1$ by Corollary \ref{cor:star}. 

Otherwise, let $D$ be an orientation of $G$. Since $\LD(D)\leq n-1$ we must actually have $\LD(D)=n-1$. Since $G$ is not at star, then $D$ must have the structure of the third condition of Theorem \ref{Digraphn-1}.

Thus, $V(G)$ can be partitioned to sets $S_1$, $C$ and $S_2$ satisfying the third condition of Theorem \ref{Digraphn-1}. Since $C$ is a bi-directed clique in Theorem \ref{Digraphn-1}, we have $|C|\leq1$ because $D$ is an oriented graph. Assume first that $|C|=1$. If $S_1$ or $S_2$ are empty, then $G$ is a star. If both of them are not empty, then $G$ contains a triangle and there is an orientation $D'$ of $G$ with an oriented cycle. Then, by Theorem \ref{Digraphn-1}, $\LD(D')\leq n-2$, a contradiction. 

If $C=\emptyset$, then $G$ is a star if either $|S_i|=1$ for $i\in \{1,2\}$ and disconnected if either is an emptyset. But if $|S_i|\geq2$, then again $G$ contains a cycle and an orientation with an oriented cycle which is against the conditions of Theorem \ref{Digraphn-1}. Hence, the claim follows.
\end{proof}

Theorem \ref{BestDirLDUppBound} ensures that, for every graph $G$, $\old(G)\leq\ld(G)$. Let us prove that if $G$ is without $C_4$ as a (not necessarily induced) subgraph, then it is actually an equality.

\begin{theorem}\label{c4freeBestLD}
Let $G$ be a  graph without $C_4$ as a subgraph. Then, $$\old(G)=\ld(G).$$
\end{theorem}
\begin{proof}
To prove this equality, let us show that, any locating-dominating set $S$ of an orientation $D$ of a graph $G$ is also a locating-dominating set for $G$. Let $D$ be an arbitrary orientation of $G$ and $S$ be a locating-dominating set of $D$. First note that $S$ is indeed a dominating set of $G$. Thus, if $S$ is not locating-dominating in $G$, then there exist $u,v\not\in S$ such that $I_G(u)=I_G(v)$. Moreover, we have $|I_G(u)|=|I_G(v)|\geq2$ since $|I_G(u)|\geq|I_D(u)|$ and $|I_G(v)|\geq|I_D(v)|$. Thus, if $|I_G(u)|=1$, then $|I_G(v)|=|I_D(v)|=|I_D(u)|=1$ and hence, $I_D(u)=I_G(u)=I_G(v)=I_D(v)$, a contradiction. Let $\{c_1,c_2\}\subseteq I_G(u)$. But then $u,c_1,v$ and $c_2$ induce a cycle on four vertices, a contradiction. 
\end{proof}

In particular, Theorem \ref{c4freeBestLD} means that $\old(T)=\ld(T)$ for any tree $T$. Let us complete this warm-up part by proving that finding the value of $\old(G)$ is NP-hard.

\medskip
\noindent\decisionpb{\PBLD}{A graph $G$, an integer $k$.}{Is it true that $\ld(G)\leq k$?}{0.5}
\noindent\decisionpb{\PBOLD}{A graph $G$, an integer $k$.}{Is it true that $\old(G)\leq k$?}{0.5}\\[0.5em]

\begin{theorem}\label{NP bestLD}
\PBLD{} and \PBOLD{} are NP-com\-plete for planar graphs of maximum degree 5 without $C_4$ as a subgraph.
\end{theorem}

\begin{proof}
Both problems are in NP. For \PBOLD, a polynomial certificate for $\old(G)\leq k$ is an orientation $D$ of $G$ and a locating-dominating set of $D$ of size at most $k$.

By Theorem~\ref{c4freeBestLD}, both values are equal in the class of graphs without $C_4$. Thus, we just prove the result for \PBLD. We reduce it from \PBD.

\medskip
\noindent\decisionpb{\PBD}{A graph $G$, an integer $k$.}{Is it true that $\gamma(G)\leq k$?}{0.5}
\medskip

We use the reduction of Gravier {\em et al.} \cite[Figure 7]{Gravier08}. Consider an instance $(G,k)$ of \PBD.  Let $G_{\triangle}$ be the graph obtained by adding to each vertex of the graph a pendant triangle (see Figure \ref{fig:reduction}).
Then it is proved in \cite{Gravier08} that $G$ has a dominating set of size $k$ if and only if $G_{\triangle}$ has a locating-dominating set of size $k+n$ (where $n$ is the number of vertices of $G$). Indeed, each triangle must contain at least one of the new vertices in a locating-dominating set and if there is exactly one vertex in a triangle, the vertex of the original graph must be dominated in the original graph.

\PBD{} has been proved to be NP-complete even for planar graphs of maximum degree 3 and girth at least 5 \cite{zvervich1995induced}. If $G$ is planar of maximum degree 3 and girth at least 5, then $G_{\triangle}$ is planar, of maximum degree 5, and does not contain $C_4$ as a subgraph. This implies our result.

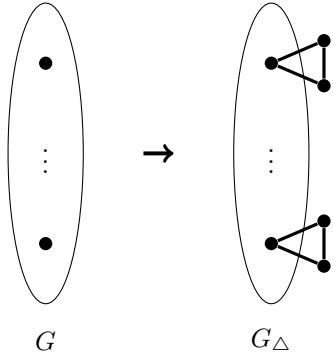
\begin{figure}
    \centering
    \begin{tikzpicture}
    \draw (0,0) ellipse (0.5cm and 2cm) node {$\vdots$};
    \draw (0,1.2) node[vertexstyle](u) {};
       \draw (0,-1.2) node[vertexstyle] {};
    \draw (0,-2.5) node{$G$};
    
    \path[draw, ->, line width=1.5pt] (1.3,0) to (1.7,0);
    \begin{scope}[shift={(3,0)}]
     \draw (0,0) ellipse (0.5cm and 2cm) node {$\vdots$};
    \draw (0,1.2) node[vertexstyle](u) {};
       \draw (0,-1.2) node[vertexstyle](v) {};
 \node[vertexstyle](tu) at (0.7,1.5) {};
 \node[vertexstyle](su) at (0.7,0.9) {};
    \node[vertexstyle](tv) at (0.7,-0.9) {};
    \node[vertexstyle](sv) at (0.7,-1.5) {};
    
    \draw[edgestyle] (u)--(tu)--(su)--(u);
     \draw[edgestyle] (v)--(tv)--(sv)--(v);
    
    \draw (0,-2.5) node{$G_{\triangle}$};
    \end{scope}
    \end{tikzpicture}
    \caption{Reduction from \PBD{} to \PBLD.}
    \label{fig:reduction}
\end{figure}

\end{proof}

\subsection{Relation to spanning subgraphs}\label{sec:subgraph}

 In this section, we prove a simple but important lemma that links $\old(G)$ with optimal locating-dominating sets of spanning subgraphs of $G$. This result is used to prove several important results all along the section, but we illustrate its interest by first giving several simple lower bounds on $\old(G)$.

\begin{lemma}\label{spanningLemma}
Let $G$ be an undirected graph. Then,
$$\old(G)=\min\{\ld(H)\mid H\text{ is a spanning subgraph of } G\}.$$
\end{lemma}
\begin{proof}
Let us show first that $\old(G)\leq\ld(H)$ holds for each spanning subgraph $H$ of $G$. Let $S$ be a locating-dominating set of a spanning subgraph $H$ of $G$. We next construct an orientation $D$ of $G$ such that an edge $e$ between $S$ and $V\setminus S$ is oriented away from the vertex in $S$ if $e\in E(H)$ and if $e\not\in E(H)$, then we orient edge $e$ towards the vertex in $S$. Other edges can be oriented in any way. Observe that we have $I_D(S;w)=I_H(S;w)$ for each vertex $w\not\in S$ and hence, $S$ is locating-dominating in $D$. Thus, $\old(G)\leq\min\{\ld(H)\mid H$ is a spanning subgraph of $G\}$.

Let us then show that for any orientation $D'$ of $G$, there exists a spanning subgraph $H'$ of $G$ such that $\ld(H')\leq \ld(D')$. Let $S$ be a locating-dominating set in $D'$. Let us construct a spanning subgraph $H'$ from the graph $G$ by having $V(H')=V(G)$ and $e=uv\in E(H')$ if and only if either $u\in S$ and the edge is oriented away from $u$ in $D'$ or $v\in S$ and the edge is oriented away from $v$ in $D'$. Observe that now $I_{D'}(S;w)=I_{H'}(S;w)$ for each vertex $w\not\in S$ and hence, $S$ is locating-dominating in $H'$. Thus, $\min\{\ld(H')\mid H'\text{ is a spanning subgraph of } G\}\leq \old(G)$ and the claim follows.
\end{proof}

In the following theorem, we apply the previous lemma on classes of graphs which are closed under (spanning) subgraphs. In particular, general lower bounds for undirected location-domination numbers in such classes also hold when we orient graphs. 

 \begin{lemma}\label{spanningClosed}
  Let $\mathcal{G}$ be a class of graphs closed under subgraphs. If there exists a function $f : \mathbb{N} \rightarrow \mathbb{N}$ such that for each graph $G\in \mathcal{G}$ with $n$ vertices we have $\ld(G)\geq f(n)$, then $$\old(G)\geq f(n).$$
 \end{lemma}
\begin{proof}
Assume by contradiction that $\old(G)<f(n)$ for some $G\in \mathcal{G}$. By Lemma \ref{spanningLemma}, there exists a spanning subgraph $H$ such that $\ld(H)=\old(G)$. So $H\in \mathcal{G}$ and $\ld(H)<f(n)$, a contradiction.
\end{proof}


As proven in \cite{rall1984location}, planar graphs satisfy $\ld(G)\geq \frac{n+10}{7}$ and outerplanar graphs satisfy $\ld(G)\geq \frac{2n+3}{7}$. Since planar and outerplanar graphs are closed under subgraphs, the following is a consequence of Lemma \ref{spanningClosed}.

\begin{corollary}\label{lowerBoundsForClassesOLD}
 Let $G$ be a planar graph on $n$ vertices. Then, $$\old(G)\geq \frac{n+10}{7}.$$
 Let $G'$ be an outerplanar graph on $n$ vertices. Then, $$\old(G')\geq \frac{2n+3}{7}.$$
\end{corollary} 

\begin{lemma}\label{DirLDLowBound}
Let $G$ be a graph of order $n$. Then, $$\old(G)\geq\frac{2n}{\Delta(G)+3}.$$
\end{lemma}
\begin{proof}
Let $G$ be a graph of order $n$. In \cite[Theorem $2$]{slater2002fault} Slater has given general lower bound $\ld(G)\geq 2n/(d+3)$ for a locating-dominating set in a $d$-regular graph $G$ on $n$ vertices. Moreover, it is easy to generalize the proof for non-regular graphs, giving $\ld(G)\geq2n/(\Delta(G)+3)$. For completeness, we include the proof here. Let $G$ be a graph on $n$ vertices with a locating-dominating set $S$. We give one \textit{share unit} for each vertex. Next, we shift $1/|I(v)|$ share from each vertex $v\in V(G)\setminus S$ to every vertex in $I(v)$. After this shift, total share over all vertices remains as $n$. Let $s$ denote the largest share in any vertex $u\in S$. Notice that $s|S|\geq n$ and hence, $|S|\geq n/s$. Moreover, we have that $s\leq 2+(\Delta(G)-1)/2$. Indeed, vertex $u$ has  share of $1$ at the beginning. After which, we shift at most $1+(\Delta(G)-1)/2$ share to $u$ since there is at most one adjacent vertex $v$ with $|I(v)|=1$. Thus, $|S|\geq  n/s\geq 2n/(\Delta(G)+3)$.

Moreover, we also have $\ld(H)\geq2n/(\Delta(H)+3)\geq2n/(\Delta(G)+3)$ for each spanning subgraph $H$ of $G$ since $\Delta(H)\leq \Delta(G)$. Thus, the claim follows from Lemma~\ref{spanningClosed} with graph class $\mathcal{G}_G=\{H\mid H \text{ is a subgraph of } G\}.$
\end{proof}

\subsection{Conjecture~\ref{conj:twinfree} holds for graph orientations}\label{sec:n/2best}

The main goal of this section is to prove Theorem \ref{thm:n/2} we restate here:

\thmtwinfreeRS*

We first need some auxiliary definitions and results.
 
Let $G=(V,E)$ be an undirected graph of order $n \geq 3$. A vertex adjacent to a leaf is called a \textit{support vertex} and a non-leaf, non-support vertex $u$ which has only support vertices as neighbours is called a \textit{support link}. The number of support vertices, leaves and support links in $G$ are denoted by respectively $s(G)$, $\ell(G)$ and $sl(G)$. Moreover, let us denote by $L(G)$, $S(G)$ and $SL(G)$ the sets of leaves, support vertices and support links, respectively, in $G$. 
By convention, for the path $P_2$ we assume that one of its two vertices is a support vertex and the other is a leaf.

We first introduce a useful lemma. The result has been previously discussed in \cite{no2007aocating1domination} and claim \ref{LDTrees_2} has been proven in \cite[Lemma $2.1$]{no2007aocating1domination}.

\begin{lemma}\label{LDTrees_support}
Let $T$ be a tree, $s \in S(T)$ with $k$ leaves $v_1,\dots,v_k$ attached to $s$. Then:
\begin{enumerate}
    \item\label{LDTrees_1} Every locating-dominating set $C$ in $T$ contains at least $k$ vertices in $\{s,v_1,\dots,v_k\}$.
    \item\label{LDTrees_2} There exists a minimum locating-dominating set $C$ in $T$ which contains all the vertices $s\in S(T)$ and for each $s\in S(T)$ there is exactly  one leaf attached to $s$ which is not in $C$.
\end{enumerate}
\end{lemma}

\begin{proof}
Let $C$ be a locating-dominating set. Let $s \in S(T)$. If $s \notin C$, then all the leaves attached to $s$ are in $C$. Otherwise, $C$ is not dominating. Let $v$ be a leaf attached to $s$. We claim that $C'=\{s\}\cup C\setminus \{v\}$ is a locating-dominating set. Indeed, we have $I(C';v)=\{s\}$ and if $I(C';u)=\{s\}$ for any $v\neq u\in V\setminus C'$, then $I(C;u)=\emptyset$. Thus, $C'$ is a locating-dominating set. So, for every $C$, there exists a locating-dominating set of the same size containing $s$. We assume that $S(T)\subseteq C$ holds in the rest of the proof.

Assume by contradiction that $|\{s,v_1,\dots,v_k\}\cap C|\leq k-1$. Since $s \in C$, there are $v_i,v_j\not\in C$ with $i \ne j$. But then $I(v_i)=I(v_j)= \{s\}$, a contradiction. So the first point holds.

Assume next that $N(s)\cap L(T)\subseteq C$. Let $v \in N(s) \cap L(T)$. Since $C$ has minimum size, there exists a vertex $u\not\in L(T)\cup C$ such that $I(u)=\{s\}$ (otherwise $v$ can be safely removed from $C$ contradicting the minimality of $C$). However, if we now consider the set $C'=\{u\}\cup C\setminus \{v\}$, then we notice immediately that $C'$ is locating-dominating and the claim follows.
\end{proof}

Locating-dominating sets in trees have been widely studied.
Blidia et al. proved in~\cite{no2007aocating1domination} that \begin{equation}\label{LDTreesOld}
\ld(T)\leq\frac{n+\ell(T)-s(T)}{2}.    
\end{equation}  Let us prove a slight improvement of this result that is needed in the proof of the main result of this section. As this is the best known upper bound for locating-dominating sets in trees, we have included a complete characterization of trees attaining it in Theorem~\ref{TreeChar}.

\begin{theorem}\label{LinkSupports}
Let $T$ be a tree of order $n\geq2$. Then, $$\ld(T)\leq \frac{n+\ell(T)-s(T)-sl(T)}{2}.$$
\end{theorem}
\begin{proof}
Let $T$ be a tree and let $F=T-SL(T)$. The set $F$ induces a forest without isolated vertices. Moreover $S(T)=S(F)$ and $L(T)=L(F)$ (by choosing the right vertex in $L$ and $S$ if the component is a $P_2$). Let $C$ be an optimal locating-dominating set in $F$ such that $S(F)\subseteq C$. Observe that now $C$ is also a locating-dominating set in $T$. Indeed, if $u\in SL(T)$, then $I(u)\subseteq S(T)$ and $|I(u)|\geq2$. Moreover, if $I(v)=I(u)$, then we have a cycle. Finally, if $u,v\not\in SL(T)$, then $I_T(u)=I_T(v)$ implies that $I_F(u)=I_F(v)$. Thus, $\ld(T)\leq |C|=\ld(F)\leq \frac{n-sl(T)+\ell(T)-s(T)}{2}$. The last inequality is due  to bound (\ref{LDTreesOld}).
\end{proof}

As a slight side-step from proving Theorem \ref{thm:n/2}, we first give a characterization for trees reaching the upper bound of Theorem \ref{LinkSupports}. For this, we need some definitions. Let $\mathcal{T}$ be a family of trees such that $T\in \mathcal{T}$ if and only if $\ld(T)=\frac{n+\ell(T)-s(T)}{2}$ where $n=|V(T)|$ or if $T=P_2$. This family has been characterized in \cite{no2007aocating1domination}. We say that trees $T_1,T_2,\dots, T_k$, where $k\geq2$ are \textit{support linked} into tree $T$ and we note $T=\mathcal{SL}(T_1, T_2,\dots, T_k)$ 
if there are vertices $v_i\in S(T_i)$ and $w\not\in \bigcup_{i=1}^{k} V(T_i)$ such that $V(T)=\bigcup_{i=1}^{k} V(T_i)\cup\{w\}$ and $E(T)=\bigcup_{i=1}^{k} E(T_i)\cup \{v_iw\mid 1\leq i\leq k\}$. Let us denote by $\mathcal{T}_{SL}$ the closure of $\mathcal{T}$ under $\SL$.

\begin{theorem}\label{TreeChar}
Let $T$ be a tree. We have $\ld(T)=\frac{n+\ell(T)-s(T)-sl(T)}{2}$ if and only if $T\in \mathcal{T}_{SL}$.
\end{theorem}
\begin{proof}
Notice first that if $sl(T)=0$, then $\ld(T)=\frac{n+\ell(T)-s(T)-sl(T)}{2}$ if and only if $T\in \mathcal{T}\subseteq\mathcal{T}_{SL}$.

Let us assume first that there exists $T\in\mathcal{T}_{SL}$ such that $\ld(T)<\frac{n+\ell(T)-s(T)-sl(T)}{2}$. Let $T$ be a tree satisfying these properties with the least number of vertices. By the previous remark, we can assume that $sl(T)>0$ and thus, that $T$ can be written as  $T=\SL(T_1,\dots, T_k)$, where $k\geq 2$, with $T_i\in \mathcal{T}_{SL}$ for both $i$. Let $w$ be the vertex in $V(T)\setminus \bigcup_{i=1}^{k} V(T_i)$ and let $v_i\in N(w)\cap V(T_i)$ for each $i\in\{1,\dots, k\}$. Notice that for any $i$, we have $v_i\in S(T)$ and $v_i\in S(T_i)$. Furthermore, by the minimality of $T$, $\ld(T_i)=\frac{|V(T_i)|+\ell(T_i)-s(T_i)-sl(T_i)}{2}$ for each $i\in\{1,\dots,k\}$. Moreover, let $C$,  be a locating-dominating set of minimum size in $T$ and $C_i$ be a locating-dominating set of minimum size in $T_i$. Notice that $$\sum_{i=1}^k\frac{|V(T_i)|+\ell(T_i)-s(T_i)-sl(T_i)}{2}=\frac{|V(T)|+\ell(T)-s(T)-sl(T)}{2}.$$ Indeed, we have $|V(T)|=1+\sum_{i=1}^k |V(T_i)|$, $\ell(T)=\sum_{i=1}^k\ell(T_i)$, $s(T)=\sum_{i=1}^k s(T_i)$ and $sl(T)=1+\sum_{i=1}^k sl(T_i)$. By Lemma \ref{LDTrees_support}, we may assume that $S(T)\subseteq C$ and $S(T_i)\subseteq C_i$ for each $i\in \{1,\dots,k\}$. 

Since $|C|< \frac{n+\ell(T)-s(T)-sl(T)}{2}$, we have $|C\cap V(T_i)|<\frac{|V(T_i)|+\ell(T_i)-s(T_i)-sl(T_i)}{2}$ for some $i\in\{1,\dots,k\}$. Since $v_i\in C\cap V(T_i)$ and since $C$ is a locating-dominating set in $T$, $C\cap V(T_i)$ is a locating-dominating set in $T_i$, a contradiction. Thus, any tree in $\mathcal{T}_{SL}$ satisfies the claim.

Let us then show that no tree outside of $\mathcal{T}_{SL}$ can satisfy the claim. Let us consider a tree $T$ of minimum size satisfying $\ld(T)=\frac{n+\ell(T)-s(T)-sl(T)}{2}$ and $T\not\in \mathcal{T}_{SL}$. Observe that $sl(T)>0$, otherwise we would have $T\in \mathcal{T}\subseteq \mathcal{T}_{SL}$. Thus, we may assume that $T=\SL(T_1,\dots, T_k)$, where $k\geq2$, for some trees $T_i$, where $T_1\not\in \mathcal{T}_{SL}$ and $w\in V(T)\setminus\bigcup_{i=1}^{k} V(T_i)$. Let $C$ be a minimum size locating-dominating set in $T$ such that $S(T)\subseteq C$ (we may assume this by Lemma \ref{LDTrees_support}). Since $S(T)\subseteq C$ and since $C$ is of minimum size, we have $w\not\in C$. Since $|V(T)|=1+\sum_{i=1}^k |V(T_i)|$, $\ell(T)=\sum_{i=1}^k\ell(T_i)$, $s(T)=\sum_{i=1}^k s(T_i)$ and $sl(T)=1+\sum_{i=1}^k sl(T_i)$ and $\sum_{i=1}^k\frac{|V(T_i)|+\ell(T_i)-s(T_i)-sl(T_i)}{2}=\frac{|V(T)|+\ell(T)-s(T)-sl(T)}{2}$, we have $|C\cap V(T_1)|=\frac{n+\ell(T_1)-s(T_1)-sl(T_1)}{2}$. Indeed, since $|C\cap V(T_i)|\leq \frac{n+\ell(T_i)-s(T_i)-sl(T_i)}{2}$, we would otherwise have $|C|<\frac{|V(T)|+\ell(T)-s(T)-sl(T)}{2} $. However, this is a contradiction on the minimality of~$T$. Thus, $T\in \mathcal{T}_{SL}$.\end{proof}

We are now ready to prove Theorem \ref{thm:n/2}.

\begin{proof}[Proof of Theorem \ref{thm:n/2}]
Let $T$ be a spanning tree of $G$ such that $\ell(T)-s(T)$ is minimal among all the spanning trees of $G$. If $T$ has $\ell(T)=s(T)$, then we are done by  Lemma~\ref{spanningLemma} and Lemma~\ref{LinkSupports}.

First, we claim that any leaf of $T$ adjacent in $T$ to a support vertex $s$ such that $|N(s)\cap L(T)|\geq2$, is adjacent, in $G$, only to vertices which are support vertices in $T$.
Observe that if $u$ and $v$ are two leaves of $T$ adjacent to the same support vertex $s$, then either $u$ or $v$ has another neighbour in $G$ since $G$ is twin-free. Moreover, if $s'\in N_G(u)$, then $s'$ is a support vertex in $T$. Indeed, if $s'$ is a leaf in $T$, then the spanning tree $T'=T-us+us'$ satisfies $\ell(T')-s(T')<\ell(T)-s(T)$, a contradiction with the minimality of $T$. Moreover, if $s'$ is a non-leaf, non-support vertex, then we have  $s(T')=s(T)+1$ and $\ell(T')=\ell(T)$, a contradiction. 

We next construct an auxiliary graph $G'$ as follows. First we add to the tree $T$ every edge $e=uv\in E(G)$ such that $u\in L(T)$, $v\in S(T)$ and there is a support vertex $s\in S(T)$  in $N_T(u)$ such that $|N_T(s)\cap L(T)|\geq 2$. Then, we delete some of the newly added extra edges so that there is exactly one leaf adjacent to every vertex in $S(T)$. The resulting graph is denoted by $G'$. Observe that, because $G$ is twin-free, none of the vertices in $L(T)$ are pairwise twins in $G'$.

Let $C'$ be an optimal locating-dominating set in $T$ such that every support vertex is included in it and for each $s\in S(T)$ there exists a leaf  $u\in N(s)\cap L(T)$ such that $u\not\in C'$. By Lemma~\ref{LDTrees_support} such a set exists. Let us now denote $C''=C'\setminus L(T)$. Now, Lemma~\ref{LinkSupports} and Lemma~\ref{LDTrees_support} together imply that $|C''|\leq n/2$. Indeed, $$|C''|=|C'|-(\ell(T)-s(T))\leq\frac{n-\ell(T)-sl(T)+s(T)}{2}.$$ Finally, we create the locating-dominating set $C$ by adding to set $C''$ all vertices in $SL(T)$ that have a twin in $G'$. Let us denote their set by $W$. Observe that if $v\in SL(T)$ has a twin $u$ in $G'$, then $v$ and  $u$ belong to a cycle in $G'$. Moreover, since $N_T(v)\subseteq S(T)$, we have $u\in L(T)$. Furthermore, vertices $u$ and $v$ may only have one twin in $G'$ and for each $s\in S(T)\cap N(u)$ we have exactly one adjacent leaf in $G'$ (which is not $u$). Thus, $\ell(T)\geq s(T)+|W|$. Hence, $|C|=|C''|+|W|\leq \frac{n-|W|-sl(T)}{2}+|W|\leq \frac{n}{2}$.

Next, we show that $C$ is a locating-dominating set in $G'$. First of all, because none of the vertices in $L(T)$ are pairwise twins in $G'$ and because $S(T)\subseteq C$, all the vertices in $L(T)$ are dominated and pairwise separated by $C$. Moreover, because we removed only leaves from $C'$, which is a locating-dominating set in $T$, and because each support vertex is in $C$, all the non-leaf vertices are dominated and pairwise separated. Finally, there is the case with $I_{G'}(C;u)=I_{G'}(C;v)$ where $u\in L(T)$ and $v\in V(T)\setminus (L(T)\cup S(T)\cup SL(T)\cup C)$. We have $|I(v)|\geq 2$, otherwise we would have $I_T(C';v)=I_T(C';u')$ for some leaf $u'\not\in C'$. Moreover, since $I(u)\subseteq S(T)$, we also have $I(v)\subseteq S(T)$. Let us denote $I(u)=\{u_1,\dots,u_t\}$, $t\geq2$, and assume without loss of generality that $uu_1\in E(T)$. Observe that because $v\not\in SL(T)\cup S(T)\cup L(T)$, there exists $w\in N_{T}(v)\setminus N_{G'}(u)$ and $w$ is not a leaf in $T$.
%
Let us next consider the tree $T''=T-u_1v+uu_2$. We notice that no new leaves are created since $\{w,u_2\}\subseteq N_{T''}(v)$ and $u_1$ has at least three neighbours in $T$, namely $v$, $u$ and at least one other leaf. Moreover, the number of support vertices does not decrease. Indeed, $u_2\in S(T'')$ and $u_1\in S(T'')$. Finally, $u\in L(T)$ but $u\not\in L(T'')$. Thus, we have $\ell(T'')-s(T'')<\ell(T)-s(T)$, a contradiction and hence, $C$ is a locating-dominating set in $G'$, a spanning subgraph  of $G$ and the claim follows by Lemma \ref{spanningLemma}.
\end{proof}

The bound $n/2$ is asymptotically tight even for graphs with large minimum degree as we can see in the next subsection (see Lemma \ref{examplemindegree}). However, it can be improved in many cases, even without the twin-freeness assumption. Let us provide two simple classes for which we can improve it.

\begin{remark}
Let $G$ be a graph on $n$ vertices with a twin-free spanning subgraph $G'$ with no isolated vertices. Then, $\old(G')\leq n/2$ by Theorem \ref{thm:n/2} and by Lemma \ref{spanningLemma}, we have $\ld(G)\leq\old(G')$. Hence, the existence of a twin-free spanning subgraph $G'$ is enough for Theorem \ref{thm:n/2} to hold.
\end{remark}



\begin{lemma}
Let $G$ be a graph on $n$ vertices with a Hamiltonian path. Then, $$\old(G)\leq \left\lceil\frac{2n}{5}\right\rceil.$$
\end{lemma}
\begin{proof}
The Hamiltonian path is a spanning subgraph. Since $\ld(P_n)=\left\lceil\frac{2n}{5}\right\rceil$ as proven in~\cite{slater1988dominating}, Lemma~\ref{spanningLemma} ensures that $\old(G)\leq \left\lceil\frac{2n}{5}\right\rceil$.
\end{proof} 

We say that a graph $G$ has a \textit{$P_{\geq t}$-factor} (or \textit{$t$-path factor}) if it has a spanning subgraph containing only paths of length at least $t$ as its components.

\begin{lemma}
Let $G$ be a claw-free graph with minimum degree $\delta\geq5t+3$, where $t$ is a positive integer, on $n$ vertices. Then, $$\old(G)\leq\frac{2t+4}{5t+8}n.$$
\end{lemma}
\begin{proof}
Let $G$ be a claw-free graph with minimum degree $\delta\geq5t+3$, where $t$ is a positive integer, on $n$ vertices. Ando et al. proved in~\cite{ando2002path} that every claw-free graph with minimum degree $\delta$ has a $P_{\geq \delta+1}$-factor.
Let $P_1,\ldots,P_q$ be the paths in the $P_{\geq\delta+1}$-factorization where $m_i=|P_i| \geq\delta+1$. As proven in~\cite{slater1988dominating}, each of these paths has a locating-dominating set of size exactly $\lceil 2m_i/5\rceil$. Hence, by Lemma \ref{spanningLemma}, we have $\old(G)\leq \sum_{i=1}^q\lceil2m_i/5\rceil=\sum_{i=1}^q(\lceil2m_i/5\rceil-2m_i/5)+\sum_{i=1}^q2m_i/5.$

Observe that we have $\lceil2m_i/5\rceil-2m_i/5\leq 4/5$ and this value is attained whenever $m_i=3\mod5$. It is easy to check that the sum is upper bounded by the case where each $m_i=5(t+1)+3$ because each $m_i\geq5t+4\geq9$ and the larger each $m_i$ is the smaller $q$ is. Hence, we have $\sum_{i=1}^q(\lceil2m_i/5\rceil-2m_i/5)+\sum_{i=1}^q2m_i/5\leq n/(5(t+1)+3)\cdot 4/5+2n/5=n(2t+4)/(5t+8)$. 
\end{proof}

\subsection{(Almost) regular graphs}\label{sec:regular}

The goal of this section is to prove that the $n/2$ bound can be drastically improved when the graph is (almost) regular. The proof is based on a probabilistic argument. Namely we prove that, if we select a random subset of vertices of the graph, then we can find an orientation where it is "almost" a locating-dominating set. That is, with positive probability, we can obtain a locating-dominating set from a random set by simply adding a small well-chosen subset of vertices to this random set.

A graph $G$ is $d$-regular if all the vertices of $G$ have degree exactly $d$. A class of graphs $\mathcal{G}$ is \emph{$k$-almost regular} if for every graph $G \in \mathcal{G}$, we have $\Delta(G) \le \delta(G)^k$.

\begin{theorem}\label{thm:regular}
Let $\mathcal{G}$ be a class of $k$-almost regular graphs. Then, there exists a constant $c_{\mathcal{G},k}$ such that, for every $G \in \mathcal{G}$, 
\[\old(G) \le c_{\mathcal{G},k} \cdot \frac{\log \delta}{\delta} \cdot n.  \]
\end{theorem}

Before proving Theorem~\ref{thm:regular}, let us make a couple of remarks. First notice that the bound is tight up to a constant multiplicative factor since, by Theorem \ref{thm:clique}, $\Theta(\log n)$ vertices are needed for cliques.

Another hypothesis of Theorem~\ref{thm:regular} asserts that there is a polynomial gap between the minimum and maximum degree. One can wonder if a similar result holds if we only have some assumptions on the minimum degree of the graph. We can prove that it is not true:

\begin{lemma}\label{examplemindegree}
Let $d,n_0 \in \mathbb N$, and $\epsilon>0$ be a real.
Then, there exists a twin-free graph $G$ of minimum degree at least $d$, order $n\geq n_0$ such that
\[ \old(G) \ge (\frac 12 - \epsilon) n \]
\end{lemma}
\begin{proof}
Let $p$ and $q$ be two integers with $p \geq q \geq 4$. 
We define the graph $G_{p,q}$ of order $n=4p+q$ as a disjoint union of $p$ paths on four vertices complete to a set $\{v_1,v_2,...,v_q\}$ of size $q$ such that the subgraph induced by $\{v_1,v_2,...,v_q\}$ is a cycle. An example is given by Figure \ref{figalmostregular}. As $p \geq q$ the minimal degree is $ \delta(G)=q+1$ and one can check $G_{p,q}$ is twin-free.\\
Let us prove that $\old(G_{p,q}) \geq 2p-2^{4q+2}$ which is enough to obtain the lemma since then, $\old(G_{p,q})/n$ will tend to $\frac 12$, when $p\to \infty$.

Let $D$ be an orientation of $G_{p,q}$ and let $S$ be an optimal locating-dominating set of $D$. Let $G_1=G_{p,q}[p_1,p_2,p_3,p_4]$,  $G_2=G_{p,q}[q_1,q_2,q_3,q_4]$ and $G_3=G_{p,q}[r_1,r_2,r_3,r_4]$ be three $P_4$ of $G_{p,q}$ which belongs to the disjoint union of $P_4$'s. If, for every $1 \leq i \leq 4$ and every $ 1 \leq j \leq q$, the edges $p_iv_j$, $q_iv_j$ and $r_iv_j$ have the same orientation in $D$, then $|S \cap V(G_1)| \geq 2$ or $|S \cap V(G_2)| \geq 2$ or $|S \cap V(G_3)| \geq 2$. Indeed, if there is at most one vertex of $S$ in each subgraph, then in each subgraph $G_i$ one extremity have no neighbour in $G_i \cap S$. Hence we can assume this is the case for $p_1$ and $q_1$. Then, $p_1$ and $q_1$ have the same neighbourhood in $S$, a contradiction. \\
There are $2^{q^4}$ orientations of edges between a set of four vertices and a set of $q$ vertices so at least $p-2 \times 2^{q^4}=p-2^{q^4+1}$ paths of the disjoint union contain at least two elements of $S$. So $\old(G_{p,q}) \geq 2p-2^{4q+2}$.
\end{proof}
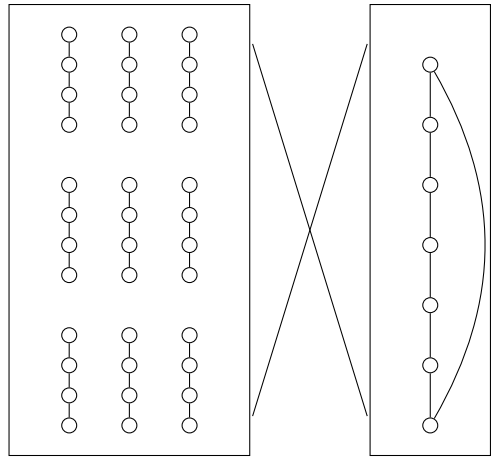
\begin{figure}[h]
    \centering
    \input{regularexample}
    \caption{Example of $G_{9,7}$ of Lemma \ref{examplemindegree}}
    \label{figalmostregular}
\end{figure}

The rest of this section is devoted to prove Theorem~\ref{thm:regular}. Let $G$ be a graph in $\mathcal{G}$. We can assume that $G$ has minimum degree at least $e^2$. (For graphs of degree less than $e^2$, the conclusion indeed follows since we can modify the constant to guarantee that $c_{\mathcal{G},k} \cdot \frac{\log \delta}{\delta} \cdot n$ is at least $n$). 
The proof is based on a probabilistic argument. We will select a subset of vertices at random and prove that, by only modifying it slightly (with high probability), we can construct an orientation of $G$ such that this set is a locating-dominating set.

Let us first recall the Chernoff inequality.

\begin{lemma}\label{lem:chernoff}[Chernoff]
Let $X= \sum_{i=1}^n X_i$ where $X_i=1$ with probability $p$ and $0$ otherwise and where all the $X_i$ are independent. Let $\mu= \mathbb{E}(X)$ and $r>0$. We have
\[ \mathbb{P}(X \le (1-r) \mu) \le e^{- \mu \cdot r^2/2}.\]
\end{lemma}
Also recall the Markov's inequality: If $X$ is a random variable taking non-negative values and $a > 0$, then:
\[ \mathbb{P}(X \geq a) \leq \frac{\mathbb{E}[X]}{a}. \] 

In order to prove Theorem~\ref{thm:regular}, we also need the following general lemma:

\begin{lemma}\label{lem:logDelta}
Let $G$ be a graph and $X$ be a subset of vertices such that every vertex $v$ not in $X$ is adjacent to at least $ \log \Delta +1$ vertices of $X$. Then, there exists an orientation $D$ of $G$ where $X$ is a locating-dominating set.
\end{lemma}
\begin{proof}
Let $V'= \{v_1,\ldots,v_t\}$ be an arbitrary ordering of $V \setminus X$. Let us prove that we can associate to each vertex $v_i$ of $V'$ a non-empty subset $S_i$ of $X \cap N(v_i)$ such that, for every $i \ne j$, $S_i \ne S_j$. 

Let us prove that such a collection of sets $S_i$ can be found greedily. Since $v_1$ is adjacent to at least $\log \Delta +1 \ge 1$ vertex of $X$, we can indeed find such a set for $v_1$. Assume that we have already selected $S_1,\ldots,S_r$. Let us prove that we can select a set for $v_{r+1}$. Let $Y_{r+1} = N(v_{r+1}) \cap X$ and $u \in Y_{r+1}$. The number of subsets of $Y_{r+1}$ containing $u$ is $2^{|Y_{r+1}|-1} \ge 2^{\log \Delta+1} \ge \Delta$. So at least one of them has not been selected since a subset $S_j$ can contain $u$ only if $v_ju$ is an edge. We arbitrarily select a subset of $Y_{r+1}$ containing $u$ that is distinct from $S_1,\ldots,S_r$, which completes the first part of the proof.

Next, for every $x \in X$ in $N(v_i)$, we orient the edges from $v_i$ to $x_j$ if $x \notin S_i$ and orient from $x$ to $v_i$ if $x  \in S_i$. One can easily check that $X$ is a locating-dominating set of this orientation of the graph.
\end{proof}

We now have all the ingredients to prove Theorem~\ref{thm:regular}.

\begin{proof}[Proof of Theorem~\ref{thm:regular}]
Let us first start with the following claim:

\begin{claim}\label{clm:goodset}
Let $c \ge 2$ be  constant. For every graph $G$ of minimum degree $\delta$, there exists a subset $X$ of $25 c \cdot (\log \delta) / \delta \cdot n$ vertices\footnote{All the logarithms of the paper have to be understood base $2$.} of $G$ such that all the vertices of $V \setminus X$  have at least $c \log \delta$ neighbours in $X$. 
\end{claim}

\begin{proof}
Start with a set $X$ which is empty and add each vertex in $X$ with probability $6c \cdot (\log \delta)/\delta$.
So $\mathbb{E}(|X|) = 6c \cdot \frac{\log \delta}{\delta} \cdot n$. Moreover $\mathbb{P}(|X| \ge 24 c \cdot \frac{\log \delta}{\delta} \cdot n) \le \frac{1}{4}$ by Markov's inequality.

Let $u$ be a vertex of $G$. Since $N(u)$ has size at least $\delta$, $\mathbb{E}(|X \cap N(u)|) \ge 6c \cdot \log \delta$. Thus, Lemma~\ref{lem:chernoff} ensures that 
\[ \mathbb{P}(|X \cap N(u) | \le c \cdot \log \delta) = \mathbb{P}(|X \cap N(u) | \le (1-5/6) 6c \cdot \log \delta)  \le e^{-6c \cdot \log \delta \cdot (5/6)^2/2} \le \frac{1}{\delta^3}   \]
as long as $c \ge 2$.

Let us next enrich $X$ with all the vertices $u$ such that $|X \cap N(u)|$ is less than $c \log \delta$. By union bound, the average number of vertices that are added in $X$ is at most $n/\delta^3$. 
Moreover, using again Markov's inequality, we know that, with probability at least $1/2$, the number of vertices that are added in $X$ is at most $ 2 \cdot n/\delta^3 \le c \cdot \frac{\log \delta}{\delta} \cdot n$. 

So, with probability at least $1/4$, the size of $X$ is at most $24 c \cdot \frac{\log \delta}{\delta} \cdot n$ before $X$ is enriched and we add at most $c \cdot \frac{\log \delta}{\delta} \cdot n$ vertices in $X$ during the second phase.
So there exists a set $X$ of size at most $25 c \log \delta / \delta \cdot n$ such that all the vertices are either in $X$ or have at least $c \log \delta$ neighbours in $X$. 
\end{proof}

Let $c=2k$. By Claim~\ref{clm:goodset}, $G$ admits a subset of vertices $X$ of size $50 k \cdot \log \delta / \delta \cdot n$ such that every vertex $v$ is either in $X$ or has at least $2k \cdot \log \delta$ neighbours in $X$. 
We claim that we can orient the edges between $X$ and $V \setminus X$ to guarantee that all the vertices of $V\setminus X$ have a different neighbourhood in $X$. 
It follows from Lemma~\ref{lem:logDelta} and the fact that $\log \Delta +1 \le 2k \cdot \log \delta$ since
$\log \Delta \le k \log \delta$. 
\end{proof}

Let us complete the results of this section with additional results on regular graphs or based on Lemma~\ref{lem:logDelta}.

A set $S$ is \textit{$k$-dominating} in $G$ if we have for each $v\in V\setminus S$ that $|N(v)\cap S|\geq k.$
Let us denote with $\gamma_k(G)$ the cardinality of a minimum $k$-dominating set of $G$. The following lemma is a simple consequence of Lemma~\ref{lem:logDelta}.

\begin{lemma}\label{lem:kdom}
Let $G$ be a graph with maximum degree $\Delta$. If $k \geq \log \Delta+1$, then $$\old(G)\leq \gamma_k(G).$$
\end{lemma}


By~\cite[Corollary 14]{delavina2010graffiti}, $ \gamma_k(G)\leq n -\alpha(G)$ while $k\leq \delta$. Then, the inequality is an immediate consequence of Lemma~\ref{lem:kdom}.

\begin{corollary}
Let $G$ be a graph with maximum degree $\Delta$ and minimum degree $\delta\geq \log \Delta+1$. Then, $$\old(G)\leq n-\alpha(G).$$
\end{corollary}
A similar result holds for locating-dominating sets, when $G$ is twin-free \cite[Corollary 4.5]{garijoconjecture}. 

\medskip

Let $M$ be a matching in a graph $G$. We say that a vertex $u\in V(G)$ is \textit{$M$-unmatched} if $u$ is not an endpoint of any edge in $M$.

\begin{theorem}\label{bestOrientdReg}
Let $G$ be a $d$-regular graph with $d\geq3$. Then, $$\old(G)\leq \alpha'(G).$$
\end{theorem}
\begin{proof}
Let $G$ be a $d$-regular graph and $M$ be a maximum matching in $G$. Moreover, let us construct  the set $D_M$ by choosing for each edge $uv\in M$ the vertex $u$ to $D_M$ if only $u$ has an adjacent $M$-unmatched vertex. If neither $u$ or $v$, or both $u$ and $v$ have an adjacent (common) $M$-unmatched vertex, then we arbitrarily add one of them to $D_M$. In the latter case, the $M$-unmatched vertex is common to $u$ and $v$ by the maximality of $M$.

Observe that $D_M$ is a dominating set in $G$ and each $M$-unmatched vertex is $2$-dominated by $D_M$. First of all, each $M$-matched vertex is dominated by another $M$-matched vertex. Secondly, no two $M$-unmatched vertices can be adjacent because $M$ is a maximum matching. Moreover, since $d\geq3$, each $M$-unmatched vertex is adjacent to the endpoints of at least two different edges in $M$. Now, due to the structure of $D_M$, each $M$-unmatched vertex is at least $2$-dominated.

Let us next construct graph $G'$ by removing each edge $e\in E(G)\setminus M$ with both endpoints in $M$-matched vertices. Now, $|I_{G'}(D_M;u)|=1$ for each $M$-matched vertex in $V(G)$ and $|I_{G'}(D_M;v)|\geq2$ for each $M$-unmatched vertex $v$. Thus, $M$-matched vertices have unique $I$-sets in $G'$. 

Let $w$ and $w'$ be two $M$-unmatched vertices with identical $I$-sets. If $2\leq|I(D_M;w)|=|I(D_M;w')|\leq d-1$, then $w$ is adjacent to vertices $u$ and $v$ with $uv\in M$ and, say, $u\in D_M$ and $v\not\in D_M$. Moreover, we also have $u\in N(w')$. But now we could have chosen $uw'$ and $vw$ in our matching $M$ which is a contradiction to the maximality of $M$.

Let us then assume that $|I(D_M;w)|=|I(D_M;w')|=d$ and $I(D_M;w)=I(D_M;w')=\{u_1,\dots u_d\}$. Thus, $w$ and $w'$ are twins. Let us then count the maximum number, $N$, of $M$-unmatched vertices which are adjacent to at least two of vertices in $I(D_M;w)$. Each vertex in $I(D_M;w)$ is adjacent in $G$ to at least one $M$-matched vertex, $u$ and $v$. Hence, there might be at most $d-3$ other adjacent $M$-unmatched vertices. Hence, we have $N\leq d(d-3)/2+2$. Furthermore, there are exactly $2^d-d-1$ subsets of $I(D_M;w)$ of cardinality at least two. Since $d\geq3$, we have $2^d-d-1>d(d-3)/2+2$. Thus, we may go through each $M$-unmatched vertex one by one and if an $M$-unmatched vertex $w$ has an $I$-set identical to some other ($M$-unmatched) vertex, then there exists a set of adjacent edges which can be removed so that $w$ has a unique $I$-set afterwards. Therefore, we may construct a spanning subgraph $G''$ with the property $\ld(G'')\leq \alpha'(G)$. Hence, the claim follows by Lemma \ref{spanningLemma}.\end{proof}

\section{Worst orientation}\label{sec:worst}
We next focus on the worst possible orientation. We again start with basic results. Then, we study the lower bound $\mold(G)\geq \ld(G)/2$ that we prove to be true for several classes of graphs and let it open in general.
Finally, we consider lower bounds using the number of vertices.

\subsection{Basic results}\label{sec:worstbasic}

Let us start by first showing some lower bounds that are used all along the section. The {\em maximum average degree} of a graph $G$, denoted by $mad(G)$ is the maximum quantity $\frac{2 |E(H)|}{|V(H)|}$ over all the subgraphs $H$ of $G$.

\begin{lemma}\label{WorstDirLDLowBound}
Let $G$ be a graph of order $n$. Then, 
\begin{enumerate}
    \item $\mold(G)\geq\alpha(G)$,
    \item $\mold(G)\geq \lceil \omega(G)/2 \rceil$,
    \item $\mold(G)\geq 2n/\lceil \text{mad}(G)/2+3\rceil$.
\end{enumerate}
\end{lemma}
\begin{proof}
Let $G$ be a graph on $n$ vertices.
Point $1$ has already been noticed for the worst orientation for dominating sets (see \cite{caro2012directed}) and thus, is still true for locating-dominating sets. We repeat here the argument. Take an independent set $X$ of size $\alpha(G)$ and orient all the edges with an endpoint in $X$ from $X$ to $V\setminus X$. Then, all the vertices in $X$ are sources and thus, must be in any locating dominating set.

Let us next prove the second point. Let $K_m$ be a clique of $G$. Consider an orientation $D$ such that each edge is oriented away from $K_m$ and the edges inside $K_m$ are oriented in a transitive way. In a locating-dominating set $S$ of $D$, no vertices outside $K_m$ can be in the in-neighbourhoods of the vertices of $K_m$. Thus, $S$ must induce a locating-dominating set in $K_m$. Since $K_m$ is oriented in a transitive way, by \cite{foucaud2020domination}, we necessarily have at least $\lceil m/2 \rceil$ vertices in $V(K_m)\cap S$ and so in $S$.  

Let us prove the last point. To do so, let us show that $\ld(D) \ge 2n / (\Delta^+(D)+3)$ for any orientation $D$ of $G$. Let $C$ be a locating-dominating set of $D$. For each vertex $c\in C$, let $s(c)=\sum_{v\in N^+[c]}1/|N^-[v]|.$ Since $C$ is dominating in $D$, we have $\sum_{c\in C}s(c)=n$. Moreover, for any $c\in C$, at most two vertices in $N^+[c]$ have only $c$ in their $I$-sets (at most one vertex outside $c$ and maybe $c$). Thus, $s(c)\leq 2+(\Delta^+(D)-1)/2$. Now, $$n=\sum_{c\in C}s(c)\leq |C|\frac{3+\Delta^+(D)}{2}.$$ Hence, $|C|\geq 2n/(3+\Delta^+(D)).$ So Point $3$ follows since each graph has an orientation $D'$ such that $\Delta^+(D')\leq \lceil \text{mad}(G)/2\rceil$ by \cite{hakimi1965degrees}.\end{proof}


Observe that all the bounds are tight. Indeed, we see in Corollary \ref{bipartiteMold} that for some bipartite graphs $\mold(G)=\alpha(G)$. Moreover, for a complete graph $K_n$, we have $\mold(K_n)=\lceil n/2\rceil$, by Corollary~\ref{cor:complete}. Finally, we will see (Corollary \ref{moldCycle}) that for a cycle on $n$ vertices we have $\mold(C_n)=\lceil n/2\rceil$.

We now present three general upper bounds for $\mold(G)$. We denote by $ad(G)$  the average degree of $G$ and by $\alpha_2(G)$ the maximum size of an independent set at $2$-distance, that is a set of vertices such that any two vertices of the set are at distance greater than 2.

\begin{lemma}\label{WorstDirLDUppBound}
Let $G$ be a graph of order $n$. Then, 
\begin{enumerate}
    \item $\mold(G)\leq n-\alpha_2(G)$;
    \item $\mold(G)\leq n -\left\lfloor\frac{\omega(G)}{2}\right\rfloor$;
    \item $\mold(G)\leq n - \left\lfloor\frac{n}{2n-2ad(G)}\right\rfloor$. 
\end{enumerate}
\end{lemma}
\begin{proof}
Let $G$ be a graph on $n$ vertices and $D$ be an orientation of $G$ such that $\ld(D)=\mold(G)$. Let $S$ be a maximum independent set at $2$-distance in $G$. Observe that, for any two distinct vertices $u,v\in S$, we have $N[u]\cap N[v]=\emptyset.$ Let us construct set $S'$ by adding, for each vertex $u\in S$, either $u$ to $S'$ if $u$ has no out-neighbours in $D$ or an out-neighbour of $u$ if $u$ has one in $D$. Now, one can easily check that $C=V\setminus S'$ is a locating-dominating set of $G$ of size $n-\alpha_2(G)$.

Let us next prove $2.$ Let $K$ be a maximal clique in $G$ and let $C_K$ be an optimal locating-dominating set in $D[K]$. Now, $C=C_K\cup (V(G)\setminus K)$ is a locating-dominating set of $D$. Furthermore, $|C_K|\leq\lceil \omega(G)/2\rceil$ by Theorem~\ref{thm:tournament} and hence, the claim follows.

Let us finally prove the third bound. We have $$\mold(G)\leq n - \lfloor\omega(G)/2\rfloor =n - \lfloor\alpha(\overline{G})/2\rfloor\leq n - \lfloor n/(2ad(\overline{G})+2)\rfloor=n - \lfloor n/(2n-2ad(G))\rfloor.$$
Here the second inequality is due to Caro-Wei lower bound for independence number \cite{caro1979new, wei1981lower}  and the last equality is due to equality $ad(G)+ad(\overline{G})=n-1$.
\end{proof}

All these bounds are tight: the first bound is tight for stars and the two others for complete graphs. 


We still have $\mold(G)\leq n-1$ as soon as $G$ has at least one edge. As in the case of $\old(G)$, we can characterize the set of graphs reaching $\mold(G)=n-1$ using Theorem \ref{Digraphn-1}.

\begin{lemma}
For a connected graph $G$, $\mold(G)=n-1$ if and only if at least one of the following conditions holds:
\begin{enumerate}
    \item $n=3$;
    \item $G$ is a star;
    \item $G$ consists of a complete bipartite graph and possibly a single universal vertex.
\end{enumerate}
\end{lemma}
\begin{proof}
By Theorem \ref{Digraphn-1}, we have $\mold(G)=n-1$ if $n=3$ or $G$ is a star. Moreover, since we consider oriented graphs, the third condition of  Theorem \ref{Digraphn-1} implies that $C$ must be of size one. Thus, the claim follows.
\end{proof}

Cycles on four vertices have a special role for best orientations. It is also the case for worst orientations, as illustrated by the following results.

\begin{lemma}\label{c4free edge removal}
Let $G$ be a graph without $C_4$ as a subgraph. Then, $\mold(G)\leq \mold(G-e)$ for any edge $e\in E(G)$.
\end{lemma}
\begin{proof}
Let $G$ be a graph without $C_4$ and with at least one edge. Let $D$ be an orientation such that $\ld(D)=\mold(G)$. By contradiction, assume that $\mold(G-e)< \mold(G)$. Then, we have $\ld(D-e)<\ld(D)$.

Let $S$ be an optimal locating-dominating set in $D-e$. Since  $\ld(D-e)<\ld(D)$, $S$ cannot be a locating-dominating set in $D$. Because $S$ is dominating in $D-e$, $S$ is also dominating in $D$. Thus, there are vertices $u,v\in V(G)$ such that $I_D(v)=I_D(u)$. Moreover, we have $|I_D(v)|=|I_D(u)|\geq2$. Let $\{c_1,c_2\}\subseteq I_D(v)$. But now we have a cycle on four vertices $u,c_1,v$ and $c_2$.
\end{proof}

Note that Lemma \ref{c4free edge removal} does not hold for $C_4$. We have $\mold(C_4)=3$ and $\mold(P_4)=2$. The bounds in the following lemma are tight for example for stars.

\begin{lemma}\label{Mold from matching}
Let $G$ be a graph without $C_4$ as a subgraph. Then, $$\ld(G)\leq\mold(G)\leq n-\alpha'(G).$$
\end{lemma}
\begin{proof}
Let $G=(V,E)$ be a graph without $C_4$ as a subgraph.  The lower bound follows from Theorem~\ref{c4freeBestLD}. Let us prove the upper bound.
Let $M$ be a maximum matching in $G$ and $G'$ be a graph we get from $G$ by removing each edge not belonging to $M$. By Lemma \ref{c4free edge removal}, we have $\mold(G)\leq \mold(G')$. Moreover, the graph $G'$ consists of isolated vertices and components isomorphic to $P_2$. Thus, a set $S$ consisting of isolated vertices and a single vertex for each $P_2$-component is locating-dominating in $G'$ and $\mold(G')=\ld(G')=|S|\leq n-\alpha'(G)$.
\end{proof}

Together with some classical results of König and Gallai, Lemma~\ref{Mold from matching} permits to determine the exact  value of $\mold(G)$ for bipartite graphs without $C_4$ (which  in particular include all trees).

\begin{corollary}\label{bipartiteMold}
Let $G$ be a bipartite graph without $C_4$ as a subgraph. Then, $\mold(G)= \alpha(G)$.
\end{corollary}
\begin{proof}
Let $G$ be a bipartite graph without $C_4$. We have $\mold(G)\geq\alpha(G)$ by Lemma~\ref{WorstDirLDLowBound}. By \cite{konig1931graphen}, we have $\alpha'(G)=\beta(G)$ since $G$ is bipartite. Moreover, by \cite{gallai1959uber}, we have  $\alpha(G)+\beta(G)=n.$ Hence, $\alpha(G)=n-\alpha'(G)$. Now, we have, by Lemma \ref{Mold from matching}, $\mold(G)\leq n-\alpha'(G)=\alpha(G)$. Thus, $\alpha(G)\leq \mold(G)\leq \alpha(G)$.
\end{proof}

\begin{corollary}\label{moldCycle}
Let $C_n$ be a cycle on $n$ vertices. Let $n=3$ or $n \geq 5$. Then, $$\mold(C_n)=\left\lceil \frac{n}{2}\right\rceil.$$
\end{corollary}
\begin{proof}
By Lemma \ref{c4free edge removal}, we have $\mold(C_n)\leq \mold(P_n)$ by Corollary \ref{bipartiteMold} applied to $P_n$ (where $\alpha(P_n)=\left\lceil \frac{n}{2}\right\rceil$). Moreover, if we take a cyclic orientation of $C_n$, the set of vertices with an odd index number forms an optimal locating-dominating set.
\end{proof}

Observe that, for the path on $n$ vertices, $P_n$, we have $\ld(P_n)=\lceil 2n/5\rceil$ for paths~\cite{slater1988dominating} while we have $\mold(P_n)=\alpha(P_n)=\lceil n/2\rceil$. As we mentioned above, there exist graphs without $C_4$ with $\mold(G)>\ld(G)$. However, we are not aware of any graph $G$ without $C_4$ which does not attain the upper bound of Lemma \ref{Mold from matching}. 

\begin{problem}
 Does there exist a graph $G$ without $C_4$ as a subgraph with $\mold(G)<n-\alpha'(G)$?
\end{problem}

\subsection{Lower bound with $\ld(G)$}\label{SecLowerboundLD}

In Section~\ref{sec:worstbasic}, we have seen that $\mold(G)\geq\ld(G)$ if $G$ is without $C_4$ subgraphs. 
One can easily remark that this equality does not hold in general. For example, for complete graphs we have $\ld(K_n)=n-1$ and $\mold(K_n)=\lceil n/2\rceil$ by Corollary~\ref{cor:complete}. However the clique example is somehow unsatisfactory since all the vertices are twins. One can wonder if we can also provide an example of twin-free graphs where $\mold(G)< \ld(G)$. We will prove (Theorem~\ref{gamma/2 example}) that there are graphs for which $\mold(G)/ \ld(G)$ is arbitrarily close to $1/2$. Moreover, we strengthen the result that $\mold(G)\geq\ld(G)$ on graphs without $C_4$ to a wider class of graphs (Lemma~\ref{lem:wormcoloring}).

Despite our efforts, we were not able to find graphs for which $\mold(G)< \ld(G)/2$. We leave  as an open problem the following question:
\begin{problem}
Is it true that for every graph $G$, $\mold(G)\geq\ld(G)/2$?
\end{problem}
We were actually not able to prove the existence of any constant $c$ such that, for any graph $G$, $\mold(G)\geq c \cdot \ld(G)$. However, in the following theorem we present a bound with $\Delta(G)$.

 \begin{theorem}
 Let $G$ be a graph. Then, $$\mold(G)\geq \frac{\ld(G)}{\lceil\log_2\Delta(G)\rceil+1}.$$
 \end{theorem}
\begin{proof}
Let $D$ be an orientation of $G$ such that $\mold(G)=\ld(D)$. Moreover, let $S$ be an optimal locating-dominating set in $D$. Observe, that for each subset $I$ of $S$, the set $$S^I=\{v\in V(G)\setminus S\mid I_G(S;v)=I\}$$ contains at most $\Delta(G)$ vertices: $|S^I|\leq \Delta(G)$. Let us next construct a new orientation $D_1$ by first taking for each set $S^I$, $\lfloor|S^I|/2\rfloor$ disjoint vertex pairs within the set $S^I$, that is, as many disjoint vertex pairs as possible. Then, we number each vertex of $V(G)$ as $u_i$, $1\leq i\leq |V(G)|$ so that each pair has consecutive numbers. Finally we orient each edge from $u_i$ to $u_j$ where $i<j$.

Let $S_1$ be an optimal locating-dominating set  for orientation $D_1$. Notice that $|S_1|\leq\mold(G)$. Moreover, $S'_1=S\cup S_1$ is a locating-dominating set in $D$ and $D_1$. Furthermore, $S_1$ separates each paired pair of vertices in $D_1$. Thus, if for a pair $u_i,u_{i+1}$, vertex $x$ separates $u_i$ and $u_{i+1}$ in $D_1$, then either $x=u_{i+1}$ or it separates also $u_i$ and $u_{i+1}$ in $G$.
Moreover, for each $I'\subseteq S'_1$ such that  $I=I'\cap S$, we have that $S_1'^{I'}\subseteq S^I$. Since $S_1$ separates the pairs in $G$, we have that $|S_1'^{I'}|\leq\lfloor|S^I|/2\rfloor\leq \lfloor\Delta(G)/2\rfloor$. 

If we now iterate this process $\lceil\log_2(\Delta(G))\rceil$ times, each time creating a new orientation with a new numbering and a new optimal locating-dominating set for the orientation, then we finally get set $S'_t=S\cup\bigcup_{i=1}^tS_i$, where $t=\lceil\log_2(\Delta(G))\rceil$, with $|S'_t|\leq \lceil\log_2(\Delta(G))+1\rceil\mold(G)$. Moreover, because we  (almost) halve the number of vertices with the same $I$-set in $G$ each time, no vertices in $V\setminus S'_t$ share the same $I$-set with the set $S'_t$ in $G$. Thus, $S'_t$ is locating-dominating in $G$ and $\ld(G)\leq|S'_t|\leq \lceil\log_2(\Delta(G))+1\rceil\mold(G).$
\end{proof}

\subsubsection{Graphs for which $\mold(G)\geq\ld(G)$}

\begin{lemma}\label{no 4-paths}
Let $G$ be a graph and $D$ be an orientation of $G$ such that no $C_4$ in $G$ contains a directed path of length $4$ in $D$. Then, any locating-dominating set of $D$ is a locating-dominating set of $G$. In particular, $\mold(G)\geq\ld(G)$.
\end{lemma}
\begin{proof}
Let $G$ be a graph and $D$ be an orientation of $G$ such that no $C_4$ in $G$ contains a directed path of length $4$ in $D$. Let $S$ be locating-dominating in $D$. Let us assume that $S$ is not locating-dominating in $G$. Set $S$ is clearly dominating in $G$. Let $v,u\in V\setminus S$ be vertices with $I_G(v)=I_G(u)$. Since $I_D(v)\neq I_D(u)$, we have $|I_G(v)|\geq2$. Let us assume that $c_1\in I_D(v)\setminus I_D(u)$ and $c_2\in I_D(u)$. But now we have a directed path $c_2uc_1v$, a contradiction.
\end{proof}

Let $M$ and $R$ be subgraphs of $G$. An \textit{$(M, R)$-WORM colouring} \cite{goddard2016vertex} of graph $G$, is a colouring of the vertices of $G$ where no subgraph of $G$ isomorphic to $M$ is monochromatic and no subgraph of $G$ isomorphic to $R$ is heterochromatic (i.e. has all its vertices of different colours). The following lemma gives us a tool for applying Lemma \ref{no 4-paths}.

\begin{lemma}\label{lem:wormcoloring}
If $G$ admits a $(K_2,C_4)$-WORM colouring, then $\mold(G)\geq\ld(G)$.
\end{lemma}

\begin{proof}
Let $G$ be a graph which admits a $(K_2,C_4)$-WORM colouring $c$ using colours $\{1,\dots,k\}$. Let $D$ be the orientation such that we have an edge from $u$ to $v$ if $c(u)<c(v)$. Since $c$ is a $(K_2,C_4)$-WORM colouring, it defines an orientation for any edge and $\mold(G)\geq\ld(D)$. Hence, it is enough to show that $\ld(D)\geq\ld(G)$.

We claim that no $C_4$ in $D$ contains a directed path of length 4. Indeed, if there is a directed path $u_1u_2u_3u_4$, then $c(u_1)<c(u_2)<c(u_3)<c(u_4)$ and if this path is contained in a $C_4$, then this $C_4$ is  heterochromatic, a contradiction. Then, the claim follows from Lemma \ref{no 4-paths}.
\end{proof}

Observe that any proper colouring with at most three colours is also a $(K_2,C_4)$-WORM colouring. Hence, we get the following corollary (where $\chi(G)$ denotes the chromatic number of $G$).

\begin{corollary}\label{WorstLD3colour}
Let $G$ be a graph with $\chi(G)\leq3$. Then, $\mold(G)\geq\ld(G)$.
\end{corollary}


\subsubsection{Worst examples}

The following theorem ensures that there exist  examples of twin-free graphs where we almost reach the ratio $\frac{1}{2}$ for $\mold(G)/\ld(G)$.

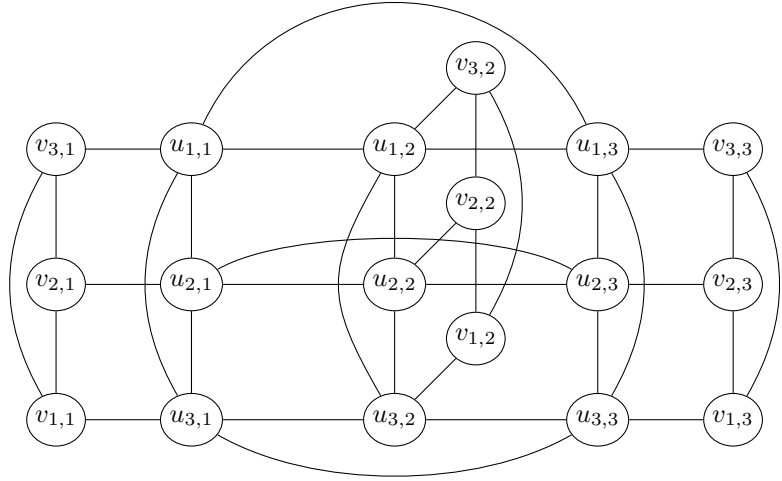
\begin{figure}[h]
    \centering
    \input{gammaHalfexample}
    \caption{Example of graph $G$ of Theorem \ref{gamma/2 example} with $t=k=3$.}
    \label{FIGHalfExample}
\end{figure}

\begin{theorem}\label{gamma/2 example}
There exists an infinite family of twin-free graphs $G$ such that $$\frac{\mold(G)}{\ld(G)}\overset{n\to\infty}{\longrightarrow}\frac{1}{2}.$$
\end{theorem}
\begin{proof}
Let $k,t\geq 2$ be integers and let $H_{k,t}$ be the graph with vertex set 
$$V(H_{k,t})=\{v_{i,j},u_{i,j}\mid 1\leq i\leq k, 1\leq j\leq t\}$$
and edge set 
$$E(H_{k,t})=\{u_{i,j}u_{i',j}\mid i\neq i'\}\cup\{u_{i,j}u_{i,j'}\mid j\neq j'\}\cup \{v_{i,j}v_{i',j}\mid i\neq i'\}\cup\{v_{i,j}u_{i,j}\}$$
where we have $1\leq i\leq k, 1\leq j\leq t$ for each $i$ and $j$. We illustrate graph $H_{3,3}$ in Figure \ref{FIGHalfExample}.

In other words, the set of vertices $\{v_{i,j}\mid 1\leq i\leq k\}$ induces a  clique $V_t^j$ for every $j$. Similarly, the set of vertices $\{u_{i,j}\mid 1\leq i\leq k\}$ induces a clique $U_t^j$ for every $j$ and the set of vertices $\{u_{i,j}\mid 1\leq j\leq t\}$ induces a clique $U_k^i$ for each $i$. In fact, the set of vertices $u_{i,j}$, for $1\leq i\leq k, 1\leq j\leq t$, forms the Cartesian product $K_t\Box K_k$. Observe that $H_{t,k}$ is twin-free since each vertex $u_{i,j}$ has a unique neighbour $v_{i,j}$ and vice versa. 

Let $C$ be a locating-dominating set of $H_{t,k}$. If we have $\{v_{i,j},u_{i,j},v_{i',j},u_{i',j}\}\cap C=\emptyset$. Then, $I(v_{i,j})=I(v_{i',j})$ and hence, we have a contradiction. Thus, $\ld(H_{t,k})\geq (k-1)t$. On the other hand, the set $\{v_{i,j}\mid 1\leq i\leq k, 1\leq j\leq t\}$ forms a locating-dominating set and hence, $$kt\geq \ld(H_{t,k})\geq (k-1)t.$$

Let us then consider the oriented locating-dominating sets. Let $D$ be an orientation of $H_{t,k}$ with $\mold(G)=\ld(D)$.

Let $S_j$ be a $2$-dominating set in the tournament $U_t^j$ for each $j$ (i.e. each vertex outside $S_j$ is dominated twice). Let $S'_i$ be a dominating set in the tournament $U_k^i\setminus \bigcup_{j=1}^tS_j$ for each $i$ in the orientation $D$. Observe that $|S_j|\leq 2\log(t+1)$ and $|S'_i|\leq \log(k+1)$ for each $i$ and $j$ by \cite{erdos1963schutte}. Moreover, let $C_j$ be an optimal locating-dominating set in the tournament $V_t^j$. We have $|C_j|\leq t/2$.

Consider next the set $C=\bigcup_{j=1}^t C_j\bigcup_{i=1}^k S'_i\bigcup_{j=1}^t S_j$. We have $$|C|\leq kt/2+2k\log(t+1)+t\log(k+1).$$
Observe that for each $i$ and $j$, vertex $u_{i,j}$ is now $3$-dominated by $\bigcup_{a=1}^k S'_a\cup\bigcup_{b=1}^t S_b$ and $I_D(u_{i,j})\neq I_D(u_{i',j'})$ where $(i,j)\neq (i',j')$. Moreover, each vertex $v_{i,j}$ is located by the set $C_j$. Thus, $C$ is a locating-dominating set of $D$ and $$\mold(H_{k,t})=\ld(D)\leq kt/2+2k\log(t+1)+t\log(k+1).$$

Finally, if we choose an orientation of $H_{k,t}$ such that each edge from $v_{i,j}$ is oriented to $u_{i,j}$ and such that all the cliques $V^j_t$ are oriented transitively, we notice that we need at least $t\lceil k/2\rceil$ vertices in $C$. 

Thus, $$\frac{\mold(H_{k,t})}{\ld(H_{k,t})}\leq \frac{kt/2+2k\log(t+1)+t\log(k+1)}{(k-1)t}$$ and $$ \frac{\mold(H_{k,t})}{\ld(H_{k,t})}\geq \frac{t\lceil k/2\rceil}{kt}\geq\frac{1}{2}.$$

When $k\to \infty$ and $t\to \infty$, $\mold(H_{k,t})/\ld(H_{k,t})\to \frac{1}{2}$. 
\end{proof}

\subsection{Lower bound with the number of vertices}\label{sec:forbidden}

In this subsection, we consider how small $\mold(G)$ can be  compared to the number of vertices. For the best orientation and the undirected case, there exist many graphs reaching the theoretical lower bound in $\Theta(\log n)$ (see Theorem \ref{thm:clique}).
For the worst orientation, we did not find any graph with
$\mold(G)$ of order $\log n$.

\begin{problem}\label{MinMold}
 Does there exist a class of graphs $\mathcal{G}$ such that for any $G\in \mathcal{G}$ on $n$ vertices the value $\mold(G)$ is logarithmic on $n$?
\end{problem}

We have three reasons to believe there is a positive answer for Open problem \ref{MinMold}. First, most of the other types of locating-dominating parameters can achieve logarithmic values on $n$. Secondly, we did not find a non-logarithmic lower bound. Thirdly, A natural class of candidates would be (Erd\H{o}s-Renyi) random graphs where an unoriented locating-dominating set has indeed logarithmic size \cite{FRIEZE}. However, the worst orientation of such a graph is not easy to manipulate and then we were not able to study efficiently upper bounds on $\mold(G)$.

On the other hand, in the following we give some properties which deny the possibility for a graph class $\mathcal{G}$ to have a logarithmic lower bound on $n$. Together with a well-known conjecture and an open problem, if they have a positive solution, these properties mean that if $\mathcal{G}$ has a certain type of a forbidden subgraph characterization, then it does not have a logarithmic lower bound for $\mold(G)$. In the following, we discuss these ideas and give some polynomial lower bounds for $\mold(G)$ in some graph classes.

Lemma \ref{WorstDirLDLowBound} gives a linear lower bound for $\mold(G)$ in $n$ for classes of graphs which have their chromatic number bounded by a  constant since $\alpha(G)\geq n/\chi(G)$ and for classes of graphs with cliques of linear size. These results can be extended to obtain bounds in $\Omega(n^\beta)$ where $\beta$ is a constant when a class of graphs $\mathcal G$ is {\em $\chi$-bounded} by a polynomial function, that is, if there exists a polynomial function $f$ such that $\chi(G)\leq f(\omega(G))$ holds for all $G\in \mathcal{G}$. Note that it has been asked \cite{karthick2016vizing} if it is true that every $\chi$-bounded class admits a $\chi$-bounding function that is polynomial. Moreover, Gy{\'a}rfas \cite{gyarfas1987problems} has conjectured that if the graph class $\mathcal{G}$ is $F$-free for some forest $F$, then $\mathcal{G}$ is $\chi$-bounded.




\begin{theorem}\label{thm:chibounded}
Let $\mathcal G$ be a class of graphs $\chi$-bounded by a function $f:x\mapsto x^c$ where $c$ is a constant.
Then, for any $G\in \mathcal G$ with $n$ vertices, we have: $$\mold(G)\geq 2^{-c/(c+1)}\cdot n^{\frac{1}{c+1}}.$$
\end{theorem}

\begin{proof}
Let $G\in \mathcal G$. By Lemma~\ref{WorstDirLDLowBound}, we have $\mold(G)\geq\omega(G)/2\geq \chi(G)^{1/c}/2$ and $\mold(G)\geq\alpha(G)\geq n/\chi(G)$. Thus, $\mold(G)\geq \max\{n/\chi(G),\chi(G)^{1/c}/2\}$. This value attains its minimum when $n/\chi(G)=\chi(G)^{1/c}/2$. In other words, when $\chi(G)=(2n)^{c/(c+1)}$. This gives the claim. 
\end{proof}

Theorem \ref{thm:chibounded} applies in particular for perfect graphs for which $f$ is the identity function. Hence, if $G$ is a perfect graph, then \begin{equation}\label{perfectBound}
    \mold(G)\geq\sqrt{\frac{n}{2}}.
\end{equation} Theorem \ref{thm:chibounded} can also be used to get a lower bound, for example, for claw-free graphs. In \cite{chudnovsky2010claw},  the authors have shown that if $G$ is a connected claw-free graph with an independent set of size at least $3$, then $\chi(G)\leq2\omega(G)$. Thus, $\mold(G)\geq \sqrt{n}/2.$ Similar idea works also for $C_3$-free graphs. In \cite{kim1995ramsey}, the author has shown that if $G$ is $C_3$-free, then $\alpha(G)\in \Omega(\sqrt{n\log n})$. Thus, also $\mold(G)\in \Omega(\sqrt{n\log n})$.

Finally, we end the chapter by giving a class of perfect graphs which shows that Bound (\ref{perfectBound}) is tight within a logarithmic multiplier. We denote by $G\square H$ the cartesian product of $G$ and $H$.



\begin{theorem}
Let $m$ be an integer. Then,  $m\leq\mold(K_m\Box K_m)\leq 3m\log(m+1)$.
\end{theorem}

\begin{proof}
Let us denote the vertices of $G=K_m\Box K_m$ by $V(G)=\{(v_i,u_j)\mid 1\leq i,j\leq m\}$. Moreover, we have $(v_{i_1},u_{j_1})(v_{i_2},u_{j_2})\in E(G)$ if $i_1=i_2$ or $j_1=j_2$. There are $2m$ cliques, each of size $m$ in $G$ and every vertex belongs to exactly two of these cliques. We have $\omega(K_m\Box K_m)=\chi(K_m\Box K_m)=m$. Thus, $m\leq\mold(K_m\Box K_m)$ and $G$ is perfect. 

Let $D$ be an orientation of $G$ such that $\mold(G)=\ld(D)$. Similarly, as in the proof of Theorem \ref{gamma/2 example}, we again construct a dominating set for each clique $\{(v_i,u_j)\mid 1\leq i\leq m\}$ where $j$ is fixed and a $2$-dominating set for each clique $\{(v_i,u_j)\mid 1\leq j\leq m\}$ where $i$ is fixed. Observe that, in $D$, each dominating set has cardinality of at most $\log(m+1)$ (\cite{erdos1963schutte}) and hence, each $2$-dominating set has cardinality of at most $2\log(m+1)$. Since we have $m$ dominating sets and $m$ different $2$-dominating sets, we have $\mold(K_m\Box K_m)\leq 3m\log(m+1)$.
\end{proof}



\bibliographystyle{abbrv}

\bibliography{LD}

\end{document}

%% file: regularexample.tex
\begin{tikzpicture}[scale=0.4]

\tikzset{circle/.style={draw,ellipse,minimum size=1mm,inner sep=2pt}}

\node[circle] (1a) at (-10,1) {};
\node[circle] (2a) at (-10,2) {};
\node[circle] (3a) at (-10,3) {};
\node[circle] (4a) at (-10,4) {};

\draw (1a) to (2a);
\draw (2a) to (3a);
\draw (3a) to (4a);

\node[circle] (1b) at (-10,6) {};
\node[circle] (2b) at (-10,7) {};
\node[circle] (3b) at (-10,8) {};
\node[circle] (4b) at (-10,9) {};

\draw (1b) to (2b);
\draw (2b) to (3b);
\draw (3b) to (4b);

\node[circle] (1c) at (-10,11) {};
\node[circle] (2c) at (-10,12) {};
\node[circle] (3c) at (-10,13) {};
\node[circle] (4c) at (-10,14) {};

\draw (1c) to (2c);
\draw (2c) to (3c);
\draw (3c) to (4c);

\node[circle] (1d) at (-12,1) {};
\node[circle] (2d) at (-12,2) {};
\node[circle] (3d) at (-12,3) {};
\node[circle] (4d) at (-12,4) {};

\draw (1d) to (2d);
\draw (2d) to (3d);
\draw (3d) to (4d);

\node[circle] (1e) at (-12,6) {};
\node[circle] (2e) at (-12,7) {};
\node[circle] (3e) at (-12,8) {};
\node[circle] (4e) at (-12,9) {};

\draw (1e) to (2e);
\draw (2e) to (3e);
\draw (3e) to (4e);

\node[circle] (1f) at (-12,11) {};
\node[circle] (2f) at (-12,12) {};
\node[circle] (3f) at (-12,13) {};
\node[circle] (4f) at (-12,14) {};

\draw (1f) to (2f);
\draw (2f) to (3f);
\draw (3f) to (4f);

\node[circle] (1g) at (-14,1) {};
\node[circle] (2g) at (-14,2) {};
\node[circle] (3g) at (-14,3) {};
\node[circle] (4g) at (-14,4) {};

\draw (1g) to (2g);
\draw (2g) to (3g);
\draw (3g) to (4g);

\node[circle] (1h) at (-14,6) {};
\node[circle] (2h) at (-14,7) {};
\node[circle] (3h) at (-14,8) {};
\node[circle] (4h) at (-14,9) {};

\draw (1h) to (2h);
\draw (2h) to (3h);
\draw (3h) to (4h);

\node[circle] (1i) at (-14,11) {};
\node[circle] (2i) at (-14,12) {};
\node[circle] (3i) at (-14,13) {};
\node[circle] (4i) at (-14,14) {};

\draw (1i) to (2i);
\draw (2i) to (3i);
\draw (3i) to (4i);

\draw (-16,0) rectangle (-8,15);
\draw (-4,0) rectangle (0,15);

\node[circle] (1) at (-2,1) {};
\node[circle] (2) at (-2,3){};
\node[circle] (3) at (-2,5){};
\node[circle] (4) at (-2,7){};
\node[circle] (5) at (-2,9){};
\node[circle] (6) at (-2,11){};
\node[circle] (7) at (-2,13){};

\draw (1) to (2);
\draw (2) to (3);
\draw (3) to (4);
\draw (4) to (5);
\draw (5) to (6);
\draw (6) to (7);
\draw (7) to[bend left] (1);

\node (i1) at (-8,14) {};
\node (i2) at (-4,1) {};
\draw (i1) to (i2);
\node (i3) at (-8,1) {};
\node (i4) at (-4,14) {};
\draw (i3) to (i4);

\end{tikzpicture}

%% file: gammaHalfexample.tex
\begin{tikzpicture}[scale=1.8]

\tikzset{circle/.style={draw,ellipse,minimum size=7mm,inner sep=0pt}}

\node[circle] (11) at (0.5,1) {$u_{3,1}$};
\node[circle] (12) at (0.5,2) {$u_{2,1}$};
\node[circle] (13) at (0.5,3) {$u_{1,1}$};
\node[circle] (21) at (2,1) {$u_{3,2}$};
\node[circle] (22) at (2,2) {$u_{2,2}$};
\node[circle] (23) at (2,3) {$u_{1,2}$};
\node[circle] (31) at (3.5,1) {$u_{3,3}$};
\node[circle] (32) at (3.5,2) {$u_{2,3}$};
\node[circle] (33) at (3.5,3) {$u_{1,3}$};

\draw (11) to (12);
\draw (11) to (21);
\draw (11) to (12);
\draw (12) to (13);
\draw (12) to (22);
\draw (13) to (23);
\draw (21) to (22);
\draw (21) to (31);
\draw (22) to (23);
\draw (22) to (32);
\draw (23) to (33);
\draw (31) to (32);
\draw (32) to (33);

\draw (11) to[bend left] (13);
\draw (21) to[bend left, looseness=1.3] (23);
\draw (31) to[bend right] (33);

\draw (11) to[bend right, looseness=0.8] (31);
\draw (12) to[bend left, looseness=0.6] (32);
\draw (13) to[bend left=65, looseness=1.2] (33);

\node[circle] (v11) at (-0.5,1) {$v_{1,1}$};
\node[circle] (v12) at (-0.5,2) {$v_{2,1}$};
\node[circle] (v13) at (-0.5,3) {$v_{3,1}$};
\node[circle] (v21) at (2.6,1.6) {$v_{1,2}$};
\node[circle] (v22) at (2.6,2.6) {$v_{2,2}$};
\node[circle] (v23) at (2.6,3.6) {$v_{3,2}$};
\node[circle] (v31) at (4.5,1) {$v_{1,3}$};
\node[circle] (v32) at (4.5,2) {$v_{2,3}$};
\node[circle] (v33) at (4.5,3) {$v_{3,3}$};

\draw (11) to (v11);
\draw (21) to (v21);
\draw (31) to (v31);
\draw (12) to (v12);
\draw (22) to (v22);
\draw (32) to (v32);
\draw (13) to (v13);
\draw (23) to (v23);
\draw (33) to (v33);

\draw (v11) to (v12);
\draw (v12) to (v13);
\draw (v21) to (v22);
\draw (v22) to (v23);
\draw (v31) to (v32);
\draw (v32) to (v33);

\draw (v11) to[bend left] (v13);
\draw (v21) to[bend right] (v23);
\draw (v31) to[bend right] (v33);

\end{tikzpicture}

%% file: Oriented_Location-Domination.bbl
\begin{thebibliography}{10}

\bibitem{ando2002path}
K.~Ando, Y.~Egawa, A.~Kaneko, K.-i. Kawarabayashi, and H.~Matsuda.
\newblock Path factors in claw-free graphs.
\newblock {\em Discrete mathematics}, 243(1-3):195--200, 2002.

\bibitem{OrientedMD}
J.~Bensmail, F.~{Mc Inerney}, and N.~Nisse.
\newblock Metric dimension: from graphs to oriented graphs.
\newblock {\em Electronic Notes in Theoretical Computer Science}, 346:111--123,
  2019.
\newblock The proceedings of Lagos 2019, the tenth Latin and American
  Algorithms, Graphs and Optimization Symposium (LAGOS 2019).

\bibitem{no2007aocating1domination}
M.~Blidia, M.~Chellali, F.~Maffray, J.~Moncel, and A.~Semri.
\newblock Locating-domination and identifying codes in trees.
\newblock {\em Australasian Journal of Combinatorics}, 39:219--232, 2007.

\bibitem{Bondy1976}
J.~A. Bondy and U.~S.~R. Murty.
\newblock {\em Graph Theory with Applications}.
\newblock Elsevier, New York, 1976.

\bibitem{caro1979new}
Y.~Caro.
\newblock New results on the independence number.
\newblock Technical report, Technical Report, Tel-Aviv University, 1979.

\bibitem{caro2012directed}
Y.~Caro and M.~A. Henning.
\newblock Directed domination in oriented graphs.
\newblock {\em Discrete Applied Mathematics}, 160(7-8):1053--1063, 2012.

\bibitem{charondirected}
I.~Charon, O.~Hudry, and A.~Lobstein.
\newblock Identifying and locating-dominating codes: Np-completeness results
  for directed graphs.
\newblock {\em IEEE Transactions on Information Theory}, 48(8):2192--2200,
  2002.

\bibitem{chudnovsky2010claw}
M.~Chudnovsky and P.~Seymour.
\newblock Claw-free graphs {VI}. colouring.
\newblock {\em Journal of Combinatorial Theory, Series B}, 100(6):560--572,
  2010.

\bibitem{cohen2018minimum}
N.~Cohen and F.~Havet.
\newblock On the minimum size of an identifying code over all orientations of a
  graph.
\newblock {\em The electronic journal of combinatorics}, P1.49, 2018.

\bibitem{delavina2010graffiti}
E.~De{L}a{V}iña, C.~E. Larson, R.~Pepper, and B.~Waller.
\newblock Graffiti.pc on the 2-domination number of a graph.
\newblock {\em Congressus Numerantium}, 203:15--32, 2010.

\bibitem{erdos1963schutte}
P.~Erd{\"o}s.
\newblock On sch{\"u}tte problem.
\newblock {\em Math. Gaz}, 47:220--222, 1963.

\bibitem{foucaudtwinfree}
F.~Foucaud, M.~A. Henning, C.~Löwenstein, and T.~Sasse.
\newblock Locating–dominating sets in twin-free graphs.
\newblock {\em Discrete Applied Mathematics}, 200:52--58, 2016.

\bibitem{foucaud2020domination}
F.~Foucaud, S.~Heydarshahi, and A.~Parreau.
\newblock Domination and location in twin-free digraphs.
\newblock {\em Discrete Applied Mathematics}, 284:42--52, 2020.

\bibitem{FRIEZE}
A.~Frieze, R.~Martin, J.~Moncel, M.~Ruszinkó, and C.~Smyth.
\newblock Codes identifying sets of vertices in random networks.
\newblock {\em Discrete Mathematics}, 307(9):1094--1107, 2007.

\bibitem{gallai1959uber}
T.~Gallai.
\newblock Über extreme {P}unkt- und {K}antenmengen.
\newblock {\em Annales universitatis scientiarum budapestinensis, Eötvös
  Sectio Mathematica}, 2:133--138, 1959.

\bibitem{garijoconjecture}
D.~Garijo, A.~González, and A.~Márquez.
\newblock The difference between the metric dimension and the determining
  number of a graph.
\newblock {\em Applied Mathematics and Computation}, 249:487--501, 2014.

\bibitem{goddard2016vertex}
W.~Goddard and H.~Xu.
\newblock Vertex colorings without rainbow or monochromatic subgraphs.
\newblock {\em arXiv preprint arXiv:1601.06920}, 2016.

\bibitem{Gravier08}
S.~Gravier, R.~Klasing, and J.~Moncel.
\newblock Hardness results and approximation algorithms for identifying codes
  and locating-dominating codes in graphs.
\newblock {\em Algorithmic Operations Research}, 3(1), 2008.

\bibitem{gyarfas1987problems}
A.~Gy{\'a}rf{\'a}s.
\newblock Problems from the world surrounding perfect graphs.
\newblock {\em Applicationes Mathematicae}, 3(19):413--441, 1987.

\bibitem{hakimi1965degrees}
S.~L. Hakimi.
\newblock On the degrees of the vertices of a directed graph.
\newblock {\em Journal of the Franklin Institute}, 279(4):290--308, 1965.

\bibitem{haynes1998domination}
T.~W. Haynes, S.~T. Hedetniemi, and P.~J. Slater.
\newblock {\em Domination in Graphs: Advanced Topics}.
\newblock Marcel Dekker, New York, 1998.

\bibitem{haynes1998fundamentals}
T.~W. Haynes, S.~T. Hedetniemi, and P.~J. Slater.
\newblock {\em Fundamentals of Domination in Graphs. 1998}.
\newblock Marcel Dekker, New York, 1998.

\bibitem{karthick2016vizing}
T.~Karthick and F.~Maffray.
\newblock Vizing bound for the chromatic number on some graph classes.
\newblock {\em Graphs and Combinatorics}, 32(4):1447--1460, 2016.

\bibitem{kim1995ramsey}
J.~H. Kim.
\newblock The ramsey number {R}(3, t) has order of magnitude $t^2$/log t.
\newblock {\em Random Structures \& Algorithms}, 7(3):173--207, 1995.

\bibitem{konig1931graphen}
D.~K{\"o}nig.
\newblock Graphen und matrizen, mat.
\newblock {\em Lapok}, 38:116--119, 1931.

\bibitem{lobstein2012watching}
A.~Lobstein.
\newblock Watching systems, identifying, locating-dominating and discriminating
  codes in graphs: a bibliography.
\newblock {\em Published electronically at \newline
  https://www.lri.fr/\%7Elobstein/debutBIBidetlocdom.pdf}.

\bibitem{rall1984location}
D.~F. Rall and P.~J. Slater.
\newblock On location-domination numbers for certain classes of graphs.
\newblock {\em Congressus Numerantium}, 45:97--106, 1984.

\bibitem{skaggs2007identifying}
R.~D. Skaggs.
\newblock {\em Identifying vertices in graphs and digraphs}.
\newblock PhD thesis, University of South Africa, 2007.

\bibitem{slater1987domination}
P.~J. Slater.
\newblock Domination and location in acyclic graphs.
\newblock {\em Networks}, 17(1):55--64, 1987.

\bibitem{slater1988dominating}
P.~J. Slater.
\newblock Dominating and reference sets in a graph.
\newblock {\em Journal of Mathematical and Physical Sciences}, 22(4):445--455,
  1988.

\bibitem{slater2002fault}
P.~J. Slater.
\newblock Fault-tolerant locating-dominating sets.
\newblock {\em Discrete Mathematics}, 249(1-3):179--189, 2002.

\bibitem{wei1981lower}
V.~Wei.
\newblock A lower bound on the stability number of a simple graph.
\newblock Technical report, Bell Laboratories Technical Memorandum 81-11217-9,
  Murray Hill, NJ, 1981.

\bibitem{zvervich1995induced}
I.~E. Zvervich and V.~E. Zverovich.
\newblock An induced subgraph characterization of domination perfect graphs.
\newblock {\em Journal of Graph Theory}, 20(3):375--395, 1995.

\end{thebibliography}
